\pdfoutput=1
\documentclass[12pt,twoside]{article}

    \usepackage{graphicx}
    \usepackage{styleset}
    \usepackage{macros}
    \let\subsubsection\subparagraph

    \title  {A deformation of instanton homology for webs}

    \author {P. B. Kronheimer and T. S. Mrowka%
      \thanks{%
        The work of the first author was supported by the National
        Science Foundation through NSF grants
        DMS-1405652 and DMS-1707924. The work of the second author was supported by
        NSF grant DMS-1406348, and by a grant from the Simons Foundation,
        grant number 503559 TSM.}}

    \address {Harvard University, Cambridge MA 02138 \\
              Massachusetts Institute of Technology, Cambridge MA 02139}

\begin{document}

\maketitle

\begin{abstract}
A deformation of the authors' instanton homology for webs is
constructed by introducing a local system of coefficients. In the case
that the web is planar, the rank of the deformed instanton homology
is equal to the number of Tait colorings of the web.
\end{abstract}

\section{Introduction}

\subsection{Statement of results}

In an earlier pair of papers \cite{KM-jsharp,KM-jsharp-triangles}, the
authors studied an $\SO(3)$ instanton homology for ``webs'' (embedded
trivalent graphs) in closed, oriented $3$-manifolds. In particular, to
a web $K\subset \R^{3}$, the authors associated a vector space
$\Jsharp(K)$ over the field $\F=\Z/2$. One of the reasons for being
interested in $\Jsharp$ is that, in conjunction with the other results of
\cite{KM-jsharp}, the following conjecture implies the four-color
theorem \cite{Appel-Haken-book}.

\begin{conjecture}[\cite{KM-jsharp}]\label{conj:tait}
    If $K$ lies in the plane $\R^{2}\subset \R^{3}$, then the
    dimension of $\Jsharp(K)$ is equal to the number of Tait colorings
    of $K$.
\end{conjecture}

The homology $\Jsharp(K)$ is constructed from the Morse theory of the
Chern-Simons functional on a space of connections $\bonf$ associated
with $K$. In this paper, we introduce a system of local coefficients
$\Gamma$ on $\bonf$ and use it to define a variant,
$\Jsharp(K;\Gamma)$, which is a module over the ring $R=\F[\Z^{3}]$
(elements of which we write as finite Laurent series in variables
$T_{1}$, $T_{2}$, $T_{3}$). The property that is conjectural for
$\Jsharp(K)$ is a theorem for its deformation $\Jsharp(K;\Gamma)$.

\begin{theorem}
    \label{thm:deformed-is-Tait}
   If $K$ lies in the plane, then the rank of $\Jsharp(K;\Gamma)$ as
   an $R$-module is equal to the number of Tait colorings of $K$.
\end{theorem}

From the construction of $\Jsharp(K;\Gamma)$ as a Morse homology with
coefficients in a local system, it is apparent that it is the homology
of a complex $(C,\partial)$ of free $R$-modules. Furthermore, the
original $\Jsharp(K$) can be recovered as the homology of the complex
$(C\otimes_{R} \F, \partial\otimes 1)$, where $\F$ is made into an
$R$-module by evaluation at $T_{i}=1$. From this it follows that there
is (for all $K$) an inequality,
\begin{equation}\label{eq:main-inequality}
          \dim_{\,\F}\Jsharp(K) \ge \rank_{R}\Jsharp(K; \Gamma),
\end{equation}
because the rank of the differential $\partial\otimes 1$ cannot be
larger than the rank of $\partial$.
 From
the theorem above, we now obtain:

\begin{corollary}
    \label{cor:Jsharp-planar-inequality}
    If $K$ lies in the plane, then the dimension of $\Jsharp(K$) is
    greater than or equal to the number of Tait colorings of $K$.
\end{corollary}

Although this corollary is ``half'' of Conjecture~\ref{conj:tait}, it
does not have any implication for the four-color theorem. It does,
however, put Conjecture~\ref{conj:tait} into perspective. The
inequality~\eqref{eq:main-inequality} can be refined by constructing a
spectral sequence, which allows us to
interpret Conjecture~\ref{conj:tait} as saying that, in the planar
case, a certain spectral sequence
collapses. For non-planar webs, the spectral sequence may not collapse,
and the inequality \eqref{eq:main-inequality} can be strict. Indeed, in
section~\ref{sec:counterexample}, we give an example of a web in
$\R^{3}$ for which $\Jsharp(K;\Gamma)$ has rank $0$ while
$\Jsharp(K)$ is non-zero. Whether or not the web is planar, the
differentials $d_{1}$, $d_{2}$ and $d_{3}$ in our spectral sequence
are always zero (as a consequence of a non-trivial calculation, Proposition~\ref{prop:d4-ss}), so
the first interesting question arises with $d_{4}$.

\subsection{Ingredients of the proof}

The content of Theorem~\ref{thm:deformed-is-Tait} is that the
deformed instanton homology can be effectively computed when $K$ is
planar. As an introduction to why this is so, recall that for each
edge $e$ of $K$ there is an operator $u_{e}$ defined on $\Jsharp(K)$
\cite{KM-jsharp} which satisfies $u_{e}^{3}=0$. We will see that there
is a similar operator on the deformed instanton homology
$\Jsharp(K;\Gamma)$ and that in the deformed case there is a relation
\begin{equation}\label{eq:u-relation}
          u_{e}^{3} + P u_{e} =0,
\end{equation} 
where $P\in R$ is the non-zero element
\begin{equation}\label{eq:p}
             P = T_{1}T_{2} T_{3} + T_{1}T_{2}^{-1}T_{3}^{-1} +
             T_{2}T_{3}^{-1}T_{1}^{-1} + T_{3}T_{1}^{-1}T_{2}^{-1}.
\end{equation}
The polynomial satisfied by $u_{e}$ therefore has two roots $0$ and
$P^{1/2}$, the second of which is repeated (because we are in
characteristic $2$). After replacing $R$ by its
field of fractions and adjoining $P^{1/2}$, we may therefore decompose
$\Jsharp(K;\Gamma)$ into the generalized eigenspaces of the operator
$u_{e}$. As $e$ runs through all edges, we obtain a collection commuting operators
whose generalized eigenspaces provide a finer decomposition of the
instanton homology. This eigenspace decomposition leads to
Tait colorings in the planar case.

The proof of the relation \eqref{eq:u-relation}, and the calculation
of $P$, involves an explicit understanding of some small instanton
moduli spaces. The particular form of $P$ is also central to the proof
that the differentials $d_{1}$, $d_{2}$, $d_{3}$ are zero in the
spectral sequence (Proposition~\ref{prop:d4-ss}), the essential point
being that $P$ vanishes to order $4$ at the point $(1,1,1)$.

\subsection{Remarks}

We have restricted our exposition in this introduction to the case that $K$ lies in $\R^{3}$ or
$S^{3}$. But the construction of $\Jsharp(K;\Gamma)$, and
the spectral sequence which leads to the 
inequality \eqref{eq:main-inequality}, both extend without
modification to the more general case of webs in a closed, oriented
$3$-manifold $Y$. We will develop the construction with this greater
generality, returning to the case of $\R^{3}$ for the proof of Theorem~\ref{thm:deformed-is-Tait}.

\section{An equivalent construction of $\Jsharp(K)$}

\subsection{Summary of the original construction}

The basic objects of study in \cite{KM-jsharp} are closed, oriented, three-dimensional
orbifolds $\check Y$ whose local isotropy groups are all either $\Z/2$
or $\Z/2\times \Z/2$. We call such an orbifold a \emph{bifold}. The
underlying topological space $|\check Y|$ is a 3-manifold, and the singular set
is an embedded trivalent graph, or web, $K\subset | \check Y|$. In the
other direction, given an
oriented 3-manifold $Y$ and a web $K\subset Y$, there is corresponding
bifold which we may denote simply by $(Y, K)$.

 By a \emph{bifold
  connection} $(E,A)$ over $\check Y$, we mean an orbifold $\SO(3)$
bundle $E$, with a connection $A$, such that the action of the isotropy groups on the fibers
of $E$ at the singular points is non-trivial. \emph{Marking
  data} $\mu$ on $\check Y$ consists of an open set $U_{\mu}$ and an
$\SO(3)$ bundle $E_{\mu}\to U_{\mu}\setminus K$. A bifold bundle $E$
is \emph{marked} by $\mu$ if there is given an isomorphism $\sigma :
E_{\mu} \to E|_{U_{\mu}}$. An \emph{isomorphism} $\tau$ between
$\mu$-marked bundles with connection, $(E,A,\sigma)$,
$(E',A',\sigma')$ is an isomorphism of bifold bundles-with-connection such that the
automorphism $\sigma^{-1}\tau\sigma' : E_{\mu}\to E_{\mu}$ lifts to
the determinant-1 gauge group. The marking data $\mu$ is \emph{strong} if the
automorphism group of every $\mu$-marked bifold connection is trivial.

Two simple examples of strong marking data are highlighted in
\cite{KM-jsharp}. The first is on the bifold $(S^{3}, H)$ where $H$ is
a Hopf link. In this example, the marking data $\mu_{H}$ has $U_{\mu}$
a ball containing $H$ and $E_{\mu}$ is a bundle with $w_{2}(E_{\mu})$
non-zero on the boundary of the tubular neighborhood of either
component of $H$. The second example is $(S^{3}, \theta)$ where
$\theta$ is an standardly embedded theta-graph. (This orbifold is the
global quotient of $S^{3}$ by an action of the Klein 4-group.) In this
example, the marking data $\mu_{\theta}$ again has $U_{\mu}$ a ball
containing $\theta$, and $E_{\mu}$ is the trivial bundle.

Given strong marking data $\mu$ on $\check Y$, there is a Banach
manifold $\bonf_{l}(\check Y;\mu)$ parametrizing isomorphism classes
of marked bundles with connection, of Sobolev class $L^{2}_{l}$. In this case, one
may define an instanton homology group $J(\check Y; \mu)$ as the Morse
homology of a perturbed Chern-Simons functional on $\bonf_{l}(\check
Y; \mu)$ with coefficients in the prime field $\F$ of characteristic
$2$. The two examples in the previous paragraph both have
one-dimensional instanton homology:
\[
\begin{aligned}
    J( (S^{3}, H) ; \mu_{H}) &= \F \\
 J( (S^{3}, \theta) ; \mu_{\theta}) &= \F.
\end{aligned}
\]

In \cite{KM-jsharp}, the authors defined $\Jsharp(\hat Y)$ for an
arbitrary bifold $\hat Y$ without marking data by forming a connected
sum with $(S^{3}, H)$, with its own marking. That is,
\[
                    \Jsharp(\check Y)\; := \;J( \check Y \csum (S^{3}, H) \; ; \; \mu_{H}).
\]
Here, when writing a connected sum, we mean that the
connected sum is formed at non-singular points of the two bifolds. The
definition is valid, because the marking data $\mu_{H}$ is strong on
the connected sum. In
the case that $\check Y$ arises from  web $K\subset \R^{3} \subset
S^{3}$, we simply write $\Jsharp(K)$.

In the general case, to make $\Jsharp$ natural, we should take $\check Y$ to be a closed bifold
with a framed basepoint in the non-singular part, in order to form the
connected sum. Having done this, $\Jsharp$ becomes a functor from a
category in which a morphism is a  4-dimensional bifold cobordism with
a framed arc joining the basepoints at the two ends of the
cobordism. As in dimension three, the underlying topological space
$|\check X|$ of a 4-dimensional bifold is a 4-manifold, and the
singular set is a restricted type of two-dimensional subcomplex called
a \emph{foam}. The space of morphisms in the category can be extended
by allowing the foams to carry extra data in the form of \emph{dots}
on the faces of the foam. See \cite{KM-jsharp}.

\subsection{Replacing the Hopf link with the theta graph}
\label{subsec:replace}

While this was not pursued in \cite{KM-jsharp},
one could form a connected sum with $(S^{3},\theta)$ instead of
$(S^{3}, H)$ above. Let us temporarily write
\begin{equation}\label{eq:Jdag-def}
                   J^{\dag}(\check Y) \; := \;J( \check Y \csum (S^{3}, \theta) \; ; \; \mu_{\theta}).
\end{equation}
We shall show that $J^{\dag}(\check Y)$ and  $\Jsharp(\check Y)$ are
isomorphic for all $\check Y$.
We then have:

\begin{proposition}
    On the
    cobordism category of bifolds with framed base-point, the two
    functors $J^{\dag}$ and $\Jsharp$ are naturally isomorphic.
\end{proposition}

\begin{proof}
    The proof is an application of an excision principle, stated in
    two variants as
    Proposition~4.1 and Proposition~4.2 in \cite{KM-jsharp}. As an
    application of the first version, a connected sum theorem is given
    as Proposition~4.3 in \cite{KM-jsharp}, and we restate here as an
    isomorphism
   \begin{equation}\label{eq:H-connected-sum}
           J( \check Y_{1} \csum \check Y_{2} \csum (S^{3}, H) ; \mu)
                       = J( \check Y_{1}  \csum (S^{3}, H)
                       ; \mu) \otimes J(  \check Y_{2} \csum (S^{3}, H) ; \mu).
   \end{equation}
   In the version described in \cite{KM-jsharp}, the marking $\mu$ was
   $\mu_{H}$ throughout. But we could also take additional marking
   data $\mu_{i}$ on $\check Y_{i}$ for $i=1,2$. So, on the left of
   the above isomorphism we could take
\[
           \mu = \mu_{1} \cup \mu_{2} \cup \mu_{H}
\]
  and so on. While this connected sum theorem is an application of
  Proposition~4.1 of \cite{KM-jsharp} and involves the Hopf link $H$,
  one can similarly use the other version, Proposition~4.2, and
  replace $H$ with $\theta$. Thus we also have
 \begin{equation}\label{eq:Theta-connected-sum}
           J( \check Y_{1} \csum \check Y_{2} \csum (S^{3}, \theta) ; \mu)
                       = J( \check Y_{1}  \csum (S^{3}, \theta)
                       ; \mu)\; \otimes\; J(  \check Y_{2} \csum (S^{3}, \theta) ; \mu).
   \end{equation}

   We now apply \eqref{eq:H-connected-sum} with $\check Y_{1} = \check
   Y$ and $\check Y_{2} = (S^{3}, \theta)$. For marking data, we take
   $\mu = \mu_{2} \cup \mu_{H}$, where $\mu_{2} = \mu_{\theta}$. From
   this we obtain
   \[
             J( \check Y \csum  (S^{3}, \theta) \csum (S^{3}, H) ;
             \mu_{\theta} \cup \mu_{H})
                       = \Jsharp(\check Y)
                         \otimes J(  (S^{3}, \theta) \csum (S^{3}, H) ;
             \mu_{\theta} \cup \mu_{H}).
   \] 
   In a similar way, from \eqref{eq:Theta-connected-sum}, we obtain
   \[
                        J( \check Y \csum  (S^{3}, \theta) \csum (S^{3}, H) ;
             \mu_{\theta} \cup \mu_{H})
                       = J^{\dag}(\check Y)
                         \otimes J(  (S^{3}, \theta) \csum (S^{3}, H) ;
             \mu_{\theta} \cup \mu_{H}).
   \]
   Combining the last two isomorphisms, we have an isomorphism of
   finite-dimensional $\F$-vector spaces of the form
   \begin{equation}\label{eq:iso-times-I}
             \Jsharp(\check Y) \otimes I = J^{\dag}(\check Y)\otimes I
    \end{equation}
   where $I = J(S^{3}, H\cup\theta; \mu)$. In the notation of
   \cite{KM-jsharp}, the vector space $I$ is $I^{\sharp}(\theta)$,
   which is non-zero by the results of Section~7.1 of
   \cite{KM-jsharp}. It follows that $\Jsharp(\check Y)$ and
   $J^{\dag}(\check Y)$ are isomorphic.

   The excision isomorphisms in \cite{KM-jsharp} are natural with
   respect to bifold cobordisms, and it follows that the isomorphism
   \eqref{eq:iso-times-I} expresses a functorial isomorphism between
   $\Jsharp$ and $J^{\dag}$.
\end{proof}

From this point on, we will use $\theta$ rather than $H$ in the
construction of $\Jsharp$. That is, we will drop the notation
$J^{\dag}$ and take \eqref{eq:Jdag-def} as the \emph{definition} of
$\Jsharp$. For brevity, we will sometimes write $\bonf^{\sharp}(\check
Y)$ for the relevant space of marked bifold connections:
\begin{equation}\label{eq:bonf-sharp}
       \bonf^{\sharp}(\check Y) = \bonf( \check Y \csum (S^{3}, \theta) ; \mu_{\theta}).
\end{equation}

\section{Instanton homology with local coefficients}

\subsection{Maps to the circle}
\label{sec:maps-to-S1}

Consider again the standard theta-graph $\theta\subset S^{3}$. The
corresponding orbifold $(S^{3}, \theta)$ is a quotient of a standard
3-sphere $\hat S^{3}$ by the Klein 4-group
$V_{4}$. Take the
marking data $\mu_{\theta}$ to have $U_{\mu} = S^{3}$ with $E_{\mu}$
trivial. (This is equivalent to the previous description, in which
$U_{\mu}$ was a ball in $S^{3}$ containing $\theta$.) 
The space of marked connections $\bonf_{\theta} = \bonf((S^{3}, \theta) ;
\mu)$ parametrizes bifold $\SO(3)$ connections $(E,A)$ equipped with a
homotopy-class of trivializations of $E$ on the non-singular part,
$S^{3}\setminus \theta$, or equivalently, a lift of the structure
group of $E$ to $\SU(2)$. If we pull back this data to $\hat S^{3}$, we
obtain a pair $(\check E, \check A)$ consisting of an $\SU(2)$ bundle
with a connection, and an action of the quaternion group  $Q_{8}$ on
the bundle, preserving the connection. The quotient $V_{4}= Q_{8} /
(\pm 1)$ acts on the base, while the kernel $\pm 1\subset Q_{8}$ acts
on the fibers of the bundle.

Let $I_{1}$, $I_{2}$ and $I_{3}$ be the non-trivial elements of
$V_{4}$, and let $\hat I_{m}$ be lifts of these in $Q_{8}$ satisfying
$\hat I_{1}\hat I_{2} = \hat I_{3}$. Let $s_{+}$ and $s_{-}$ be the
two vertices of the theta graph, and let $\hat s_{\pm}$ be their
unique preimages in $\hat S^{3}$  Let $\gamma_{m}$ ($m=1,2,3$) be
the arcs of $\theta$ where the isotropy group is generated by $I_{m}$,
and let $\hat \gamma_{m}$ be chosen lifts of these, as paths from
$\hat s_{-}$ to $\hat s_{+}$. The lifts can be chosen so that their
tangent vectors at $\hat s_{-}$ are a right-handed triad. 

Let $(\hat E, \hat A)$ be  an equivariant bundle on $\hat S^{3}$,
as above. Let $\hat E_{+}$ and $\hat E_{-}$
be the fibers of $\hat E$ over $\hat s_{\pm}$. Parallel transport
along $\hat{\gamma}_{m}$ gives an isomorphism 
\[
\iota_{m}(\hat A) : \hat E_{-} \to \hat E_{+}.
\]
Since $\hat \gamma_{m}$ is fixed by $\hat I_{m}$, the isomorphism
$\iota_{m}(\hat A)$ commutes with the action of $\hat I_{m}$ on $\hat
E_{-}$ and $\hat E_{+}$. Let
\[
\begin{aligned}
    \tau_{-} : \hat E_{-} & \to \C^{2} \\
    \tau_{+} : \hat E_{+} & \to \C^{2} 
\end{aligned}
\]
be $Q_{8}$-equivariant isomorphisms of determinant $1$. Each of these
is unique up to $\pm 1$. The isomorphism
\[
             \tau_{+}\circ \iota_{m}(\hat A) \circ \tau_{-}^{-1} : \C^{2} \to \C^{2}
\]
is then an element of $\SU(2)$ that commutes with $\hat I_{m}$, and
which has an ambiguity of $\pm 1$ resulting from the ambiguity in
$\tau_{\pm}$. The commutant $C(\hat I_{m})$ in $\SU(2)$ is a copy of
the circle group, and can be identified with the standard circle
$S^{1} = \R/  \Z$ by a unique isomorphism
\[
           L : C(\hat I_{m}) \to \R/  \Z
\]
sending $\hat I_{m}$ to $1/4$. To remove the ambiguity of $\pm 1$, we
take the square (written as $2L$ in additive notation), and we define
\[
            h_{m} (\hat A) = 2L ( \tau_{+}\circ  \iota_{m}(\hat A)\circ  \tau_{-}^{-1}),
\] 
to obtain a well-defined map
\begin{equation}\label{eq:h-maps}
           h_{m} : \bonf_{\theta} \to \R/\Z, \qquad m=1,2,3.
\end{equation}

The maps $h_{1}$, $h_{2}$, $h_{3}$ depend on some universal choices
(essentially the labeling of some standard arcs in $\hat S^{3}$), but
we regard these choices as having been made, once and for all.

The maps $h_{m}$ constructed here for the orbifold
$(S^{3}, \theta)$ can be defined in the same way 
for connected sum with $(S^{3}, \theta)$. That is, if we take an
arbitrary bifold $\check Y$ and consider the space of marked
connections $\bonf^{\sharp}(\check Y)$ defined as in
\eqref{eq:bonf-sharp}, then we have maps
\begin{equation}\label{eq:h-maps-Y}
           h_{m} : \bonf^{\sharp}(\check Y) \to \R/\Z, \qquad m=1,2,3,
\end{equation}
defined in just the same way, using $\theta$.

\subsection{The local system}

Given a topological space $B$ and a continuous map
\[
         h : B \to \R / \Z.
\]
we can construct a local system $\Gamma$
of rank-1 $R$-modules $\Gamma$ on $B$, where $R$ is
the ring $\F[\Z]=\F[T,T^{-1}]$ of finite Laurent series. The local
system we want is pulled back from the universal example on $\R/\Z$. We do this
concretely by regarding $R$ as contained in the larger ring $\F[\R]$,
and for each $b\in B$ we set
\[
            \Gamma_{b} = T^{\tilde h(b)} R
\]
where $\tilde h(b)$ is any lift of $h(b)$ to $\R$. If $\zeta$ is a
path from $a$ to $b$, then
\[
          \Gamma_{\zeta} : \Gamma_{a} \to \Gamma_{b}
\]
is multiplication by $T^{\tilde h(b) - \tilde h(a)}$, where $\tilde h$
is any continuous lift of $h$ along the path.
Given a collection of maps $h_{1}, \dots, h_{n}$ from $B$ to
$\R/\Z$ we can similarly construct a local system of rank-1
$R$-modules, where $R$ is now the ring of finite Laurent series in $n$
variables $T_{1}, \dots, T_{n}$. 

We apply this now to these three circle-valued maps $h_{1}$, $h_{2}$,
$h_{3}$, defined above at \eqref{eq:h-maps-Y}, to obtain a local system
$\Gamma$ on $\bonf^{\sharp}(\check Y)$ for any bifold $\check Y$. This
is a local system of free $R$-modules of rank $1$, where
$R = \F[\Z^{3}]$ is the ring of finite Laurent series in variables
$T_{1}, T_{2}, T_{3}$. (Local systems of this sort were constructed in
\cite{KM-singular} and \cite{KM-s-invariant}, using maps of a similar
sort, but see the remarks at the end of this section for a brief discussion.)

We can now construct the instanton homology groups $\Jsharp(\check Y ;
\Gamma)$ with coefficients in the local system $\Gamma$ in the usual
way. That is, after equipping $\check Y$ with an orbifold Riemannian
metric and perturbing the Chern-Simons functional on
$\bonf^{\sharp}(\check Y)$ to achieve
Morse-Smale transversality, we define a differential $R$-module
\begin{equation}\label{eq:Csharp}
                    C^{\sharp}(\check Y ; \Gamma) = \bigoplus_{\beta} \Gamma_{\beta}
\end{equation}
where the sum is over critical points (a finite set), with boundary
map
\begin{equation}\label{eq:partial-def}
       \partial_{\,\Gamma} = \sum_{(\alpha,\beta,\zeta)}  \Gamma_{\zeta}
\end{equation}
where the sum is over pairs of critical points $(\alpha,\beta)$ and
gradient trajectories $\zeta$ of index $1$. The finitely-generated $R$-module
$\Jsharp(\check Y; \Gamma)$ is defined as the homology of this
complex. We summarize the definition:

\begin{definition}\label{def:deformed-jsharp}
    For a closed oriented bifold $\check Y$, we define
    $\bonf^{\sharp}(\check Y)$ as the space of marked $\SO(3)$
    connections on the marked bifold
    $(\check Y \# (S^{3}, \theta) \, ; \,\mu_{\theta})$. Here the
    marking region of $\mu_{\theta}$ is an open ball $B^{3}$
    containing the graph $\theta\subset S^{3}$, with the trivial
    bundle. We write $R$ for the ring $\F[\Z^{3}]$ of finite Laurent
    series in variables $T_{1}$, $T_{2}$, $T_{3}$, and we define
    $\Gamma$ as the local system of $R$-modules on
    $\bonf^{\sharp}(\check Y)$ constructed from the three
    circle-valued functions \eqref{eq:h-maps-Y}. We define $J^{\sharp}(\check Y; \Gamma)$
    as the Floer homology with these local coefficients on
    $\bonf^{\sharp}(\check Y)$. \qed
\end{definition}

As usual, given a web $K\subset Y$, where $Y$ is a 3-manifold with
framed basepoint, we write
$\Jsharp(Y, K ; \Gamma)$ for the case that $\check Y = (Y, K)$. If
$Y=S^{3}$ and the framed basepoint is at infinity, we simply write
$\Jsharp(K; \Gamma)$.

We can give a concrete interpretation of the maps $\Gamma_{\zeta}$ for
a path $\zeta$ from $\alpha$ to $\beta$ in $\bonf^{\sharp}(\check Y)$,
which will be helpful later. First, our definition means that
\begin{equation}\label{eq:Gamma-explained}
          \Gamma_{\zeta} = T_{1}^{\Delta \tilde h_{1}} 
                  T_{2}^{\Delta \tilde h_{2}}
          T_{3}^{\Delta\tilde h_{3}}
\end{equation} 
where $\Delta \tilde h_{m}$ is the change in a continuous lift $\tilde
h_{m}$ of $h_{m}$ along the path $\zeta$. Since $h_{m}$ is defined in
terms of the
holonomy of an $S^{1}$ connection, the change in $h_{m}$ along a path
can be expressed as the integral of the curvature of an $S^{1}$
bundle. If $\gamma_{m}$ ($m=1,2,3$) are again the three arcs of
$\theta$ and $\hat \gamma_{m}$ their chosen lifts to $\hat S^{3}$. The
path $\zeta$ gives rise to an $\SO(3)$ connection on $\R\times \hat
S^{3}$ which we may restrict to $\R\times \hat \gamma_{m}$, where its
structure group reduces to $C(I_{m})_{1}$ (the identity component of
the commutant of $I_{m}$ in $\SO(3)$). The latter is a circle group,
so there are two possible isomorphisms
\[
         C(I_{m})_{1} \to S^{1} \subset \C.
\]
The marking gives us a preferred lift $\hat I_{m}$ in $\SU(2)$ and
picks out a preferred isomorphism of $C(I_{m})_{1}$ with $S^{1}$. Via
this isomorphism, our bundle on $\R\times \hat\gamma_{m}$ becomes an
$S^{1}$ bundle $K$. The change in the holonomy, $\Delta \tilde h_{m}$, 
is then then Chern-Weil
integral:
\begin{equation}\label{eq:Chern-Weil-h}
            \Delta \tilde h_{m} = \frac{i}{2\pi} 
           \int_{\R\times \hat \gamma_{m}} F_{K}.
\end{equation}

\begin{remarks}
     As mentioned above, instanton Floer homology with local
     coefficients was applied previously to knots and links in the
     authors' earlier papers \cite{KM-singular} and
     \cite{KM-s-invariant}. The local systems used there (also denoted
     by the generic letter $\Gamma$) were  defined in a similar
     manner, but the circle-valued functions $h$ that were used were
     obtained from holonomies along the components of the knot or link
     $K$ itself. By contrast, the circle-valued functions in the
     present paper are defined using the holonomy along the edges of
     the auxiliary graph $\theta$.
\end{remarks}

\subsection{Functoriality and basic properties}

The functorial properties of $\Jsharp(Y, K)$ carry over to the case of
local coefficients without essential change. As in \cite{KM-jsharp}
(see Definition~3.10),  we define a category
$\mathcal{C}^{\sharp}$ whose objects are pairs $(Y,K)$ where $K$ is a
web in a closed, oriented $3$-manifold $Y$ equipped with a framed
basepoint $y_{0} \in Y\setminus K$. The morphisms are isomorphism
classes of triples $(X,\Sigma, \gamma)$, where $X$ is 4-dimensional
cobordism, $\Sigma$ is an embedded foam with ``dots'', and $\gamma$ is
a framed arc joining the basepoints. Then we have a functor
\[
      \Jsharp( \; \text{---} \; ; \Gamma) : \mathcal{C}^{\sharp} \to
      (\text{finitely-generated $R$-modules}) 
\]
where $R=\F[\Z^{3}]$. To the empty web in $\R^{3}\subset S^{3}$, for example, this functor
assigns the free rank-1 module $R$; and to a closed foam in $\R^{4}$
it assigns a module map $R\to R$, i.e. an element of $R$ itself.

We have the usual applications of excision. So for example, as with
\cite[Corollary 4.4]{KM-jsharp}, we have the following multiplicative
property.

\begin{proposition}
    \label{prop:tensor-product}
     Suppose $K=K_{1}\cup K_{2}$ is a split web in
     $\R^{3}$, meaning that there is an embedded $2$-sphere $S$ which
     separates $K_{1}$ from $K_{2}$. Suppose that at least one of
     $\Jsharp(K_{i};\Gamma)$ is a free $R$-module. Then there is an isomorphism,
\[
         \Jsharp(K;\Gamma) = \Jsharp(K_{1};\Gamma) \otimes_{R} \Jsharp(K_{2};\Gamma).
\]
    Moreover, if $\Sigma\subset [0,1]\times S^{3}$ is a split cobordism, meaning that
    $\Sigma=\Sigma_{1}\cup \Sigma_{2}$ and $\Sigma$ is disjoint from
    $[0,1]\times S$, then
\[
         \Jsharp(\Sigma) = \Jsharp(\Sigma_{1})\otimes \Jsharp(\Sigma_{2}).
\] 
    In general, if neither is free, then $\Jsharp(K;\Gamma)$ is
    related to $ \Jsharp(K_{1};\Gamma)$ and $\Jsharp(K_{2};\Gamma)$ by
    a K\"unneth theorem (a spectral sequence).\qed
\end{proposition}

The other general properties that we wish to restate in the
local-coefficient case are the exact triangles from
\cite{KM-jsharp-triangles}. The setup here is that we have six webs
(in $\R^{3}$ or in a general $3$-manifold) which are the same outside
a ball and which differ inside the ball as shown in
Figure~\ref{fig:skein}.
\begin{figure}
    \begin{center}
        \includegraphics[scale=.5]{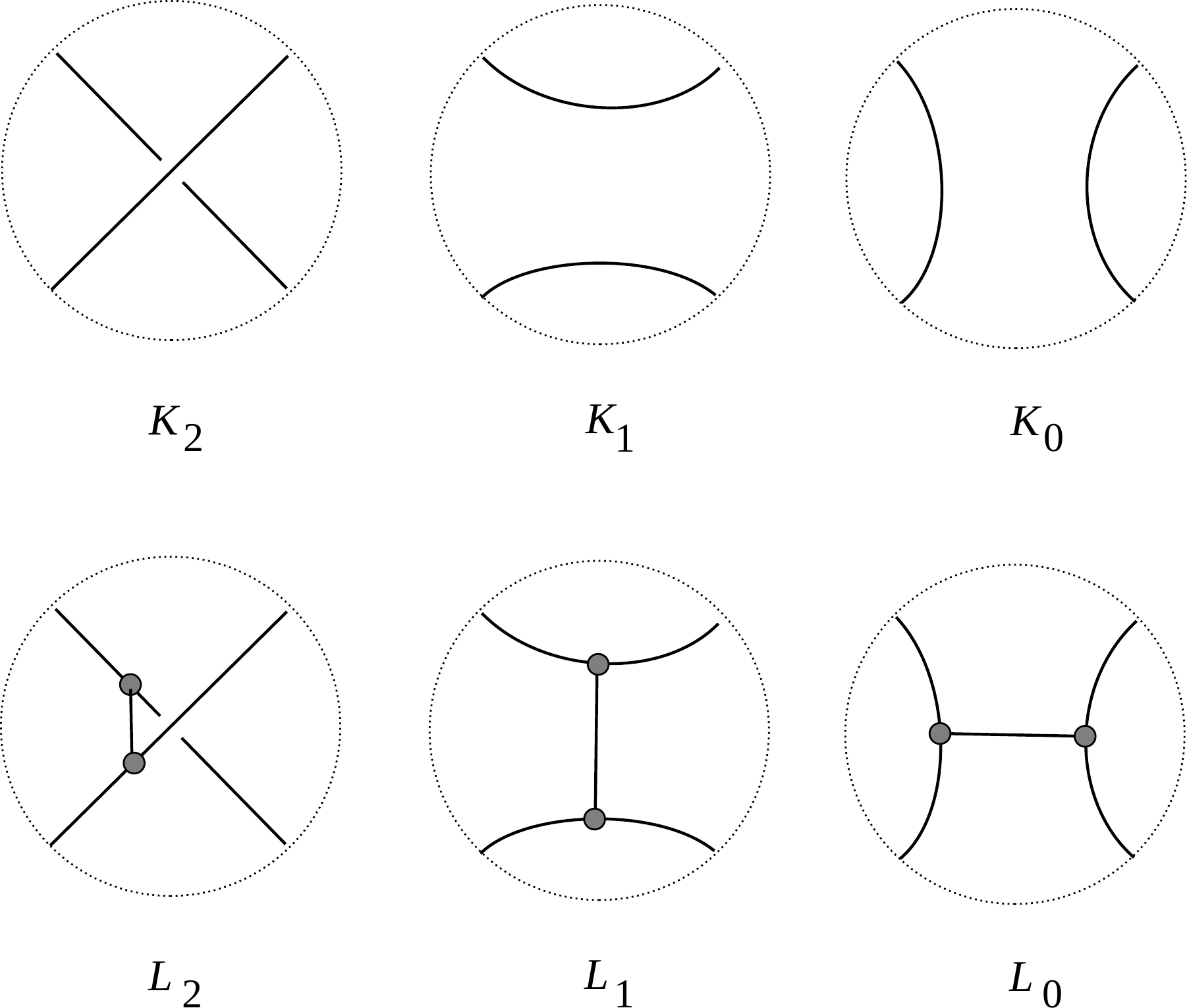}
    \end{center}
    \caption{\label{fig:skein}
    Six webs in $Y$ differing inside a ball.}
\end{figure}
There are standard elementary foam-cobordisms from $K_{i+1}$ and $L_{i+1}$ to $K_{i}$ and
to $L_{i}$.

\begin{proposition}
    \label{prop:exact-triangle}
The sequence of $R$-modules obtained by applying the functor $\Jsharp(\;
\text{---}\; ; \Gamma)$ to the 3-periodic sequence
\begin{equation}\label{eq:LLL-exact}
        \dots \to L_{2} \to L_{1} \to  L_{0} \to L_{2} \to \cdots      
\end{equation}
is exact. So too is the sequence obtained from
\begin{equation}\label{eq:LKK-exact}
       \dots \to L_{2} \to K_{1} \to  K_{0} \to L_{2} \to \cdots 
\end{equation}
as well as the two other sequence obtained by cyclically permuting the
indices (though there is no essential difference).
\end{proposition}

\begin{proof}
    The proof for the sequence \eqref{eq:LLL-exact} is given in
    \cite{KM-jsharp-triangles}, and carries over essentially without
    change. For the second sequence \eqref{eq:LKK-exact}, a short-cut was taken in
    \cite{KM-jsharp-triangles}: the exactness of \eqref{eq:LKK-exact}
    was deduced from the exactness of  \eqref{eq:LLL-exact}, using additional
    relations satisfied by $\Jsharp$. Nevertheless, as indicated in
    \cite{KM-jsharp-triangles}, one can ignore the short-cut and give
    a direct proof for \eqref{eq:LKK-exact}, along just the same
    lines as  \eqref{eq:LLL-exact}. This direct proof adapts without essential
    change to the case of local coefficients.
\end{proof}

\section{The cubic relation for $u$}

\subsection{Statement of the result}

To each edge $e$ of a web $K \subset Y$, there is an operator
\[
              u_{e} : \Jsharp( Y, K ; \Gamma) \to \Jsharp( Y, K ; \Gamma).
\]
In the language of foams with dots, this is the operator corresponding
to a cylindrical foam $[0,1]\times K$ with a dot on the face
$[0,1]\times e$. With constant coefficients $\F$, the operator
satisfies $u_{e}^{3} = 0$, as shown in \cite{KM-jsharp}. 
But the result with the local system of
coefficients $\Gamma$ is different.

\begin{proposition}\label{prop:u-relation}
    The operator $u_{e}$ satisfies 
\[
             u_{e}^{3} + P u_{e} = 0,
\]
where $P\in R$ is the element \eqref{eq:p}.
\end{proposition}

This result is of a very similar form to the $N=3$ case of a general
result for $\SU(N)$ gauge theory, treated in \cite{Xie}. Our argument
follows \cite{Xie} quite closely, though in the end the argument here is
considerably simpler. The proof has some setup required, which we
present first.

\subsection{Characteristic classes of the base-point bundles}

Let $\check X$ be a connected 4-dimensional bifold cobordism from $\check Y'$
to $\check Y$. Let $\nu$ be marking data on $\check X$ (possibly
empty) and let $\bonf^{*}(\check X;\nu)$ denote the space of marked bifold
connections which are fully irreducible (i.e. have trivial
automorphism group). Depending on the bundle and the marking, this may
be all of $\bonf(\check X ;\nu)$. Write $\check X^{o}$ for the non-singular
part of the orbifold. Then there is a universal $\SO(3)$ bundle 
\[ \mathbb{E}
\to \check X^{o} \times \bonf^{*}(\check X;\nu).
\]
If we pick a basepoint $x\in \check X^{o}$, then we obtain a bundle
$\mathbb{E}_{x}\to \bonf^{*}(\check X;\nu)$ and characteristic classes
\[
            \mathsf{w}_{i,x} = w_{i}(\mathbb{E}_{x}) \in H^{i}(\bonf^{*}(\check X;\nu); \F).
\]

If $s$ belongs to a 2-dimensional face of the singular set of
$\check{X}$, then we can pass to a smooth $\Z/2$ cover of a chart
around $s$ and obtain a bundle
\[
        \tilde {\mathbb{E}}_{s} \to \bonf^{*}(\check X;\nu)
\]
carrying an involution on the fibers. The $+1$ and $-1$ eigenspaces of the
involution are respectively a line bundle
\begin{equation}\label{eq:bpoint-line}
      \mathbb{L}_{s} \to \bonf^{*}(\check X;\nu)
\end{equation}
and a rank-$2$ bundle
\begin{equation}\label{eq:bpoint-W}
            \mathbb{W}_{s} \to \bonf^{*}(\check X;\nu).
\end{equation}
Although phrased differently, the real line bundle $ \mathbb{L}_{s}$
is the same one that is used
to define the operators $u$ in \cite{KM-jsharp}, and we write
\[
           \mathsf{u}_{s} = w_{1}(\mathbb{L}_{s}) \in
           H^{1}(\bonf^{*}(\check X;\nu)).
\]

\begin{lemma}\label{lem:u-relation-H}
    We have the relation
\[
              \mathsf{u}_{s}^{3} + \mathsf{w}_{1,x} \mathsf{u}_{s}^{2}
 + \mathsf{w}_{2,x} \mathsf{u}_{s} + \mathsf{w}_{3,x} = 0.
\]
\end{lemma}

\remark {Notwithstanding the notation in the lemma, the class
  $\mathsf{w}_{1,x}$ is zero, because our bundles have structure group
  $\SO(3)$ rather than $O(3)$.}

\begin{proof}[Proof of the lemma.]
  Although our exposition asked that $x$ be in the smooth part, we may
  use the $\Z/2$ cover of a coordinate chart at $s$ in the singular
  set of $\check X$, and we can therefore compute the classes
  $\mathsf{w}_{i}$ as
\[
            \mathsf{w}_{i,x} = w_{i} (\tilde{ \mathbb{E}}_{s}),
\]  
 where $\tilde{\mathbb{E}}_{s}$ is the bundle which has
 $\mathbb{L}_{s}$ as a subbundle. The relation in the lemma is a
 universal relation for a rank-3 bundle containing a line subbundle:
 it expresses the fact that the complementary rank-2 bundle has $w_{3}=0$.
\end{proof}

\subsection{Properties of the point-class operators}
\label{sec:op-props}

Next we parlay the relation from Lemma~\ref{lem:u-relation-H}  into a
relation among operators on the instanton Floer homology groups. We
recall how certain cohomology classes in the space of connections give
rise to operators, following constructions that go back to
\cite{Donaldson-polynomials}.

To put ourselves in the framework for $\Jsharp(\check Y; \Gamma)$, 
we take $\check X$ to be the product
cobordism from $\check Y \csum (S^{3},\theta)$ to itself, and we write $\check X^{+}$ for the
infinite cylinder. We take the framed arc $\gamma$ in $\check X$
joining the basepoints to be $[0,1]\times y_{0}$, and the strong
marking data $\nu$
also to be a product $[0,1]\times \mu_{\theta}$. We suppose that a
holonomy perturbation of the Chern-Simons functional has been chosen
so that the critical points are non-degenerate and the moduli spaces
$M(\alpha, \beta)$ of marked solutions on $\check X^{+}$ are cut out
transversely.

Now let $Z\subset \check X$ be compact 4-dimensional sub-bifold with
boundary. By its intersection with $\nu$, this bifold inherits marking
data, and we suppose that the restriction map
\[
          H^{1}(U_{\nu}; \F) \to   H^{1}(U_{\nu}\cap Z; \F)
\]
is injective. This is sufficient to ensure that if $(E,A)$ is a
solution in any of the moduli spaces $M(\alpha, \beta)$, then the
restriction of $(E,A)$ to $Z$ has trivial automorphism group. We
therefore have restriction maps
\begin{equation}\label{eq:M-restriction}
                  M(\alpha, \beta) \to \bonf^{*}(Z; \nu \cap Z).
\end{equation}
Let $\mathsf{v}\in H^{d}(  \bonf^{*}(Z; \nu \cap Z)\; F)$ be a cohomology
class represented as the dual of a codimension-$d$ subvariety $V$,
which we take to be subset stratified by Banach submanifolds with a smooth
open stratum of codimension $d$ and other strata of codimension $d+2$
or more. We suppose these are chosen so that all strata are transverse
to the maps \eqref{eq:M-restriction}.

The intersections $M_{d}(\alpha, \beta) \cap V$ are compact by
Uhlenbeck's theorem provided that either:
\begin{enumerate}
\item $d\le 7$ and $Z$ does not meet the singular set of the orbifold; or
\item $d\le 3$ and $Z$ meets the singular set of $\check X$ only in
    the strata with $\Z/2$ stabilizer.
\end{enumerate}

If $d \le 6$ or $d \le 2$ respectively in the above cases, then we
obtain a chain map by summing over the points in these compact
intersections. (The proof of the chain property involves moduli spaces
of dimension $d+1$, which is the reason for the stronger restriction
on $d$.) The chain map on $C^{\sharp}(\check Y; \Gamma)$ is
defined by
\[
       \sum_{\alpha,\beta}\;  \sum_{\zeta \in M_{d}(\alpha, \beta)\cap
       V}
       \Gamma_{\zeta},
\]
and it gives rise to an operator
\[
          v : \Jsharp(\check Y; \Gamma) \to \Jsharp(\check Y; \Gamma).
\]

In this way, the cohomology class $\mathsf{u}_{s}$, for $s$ a
basepoint in a face $[0,1]\times e$ of the singular set, gives rise to
the operator
\[
         u_{e} :  \Jsharp(\check Y; \Gamma) \to \Jsharp(\check Y; \Gamma),
\]
and the classes $\mathsf{w}_{i,x}$ give rise to operators
\[
          w_{i,x} :  \Jsharp(\check Y; \Gamma) \to \Jsharp(\check Y; \Gamma)
\]
for $i=1,2,3$. 

\begin{lemma}\label{lem:proto-u-relation}
    The above operators satisfy the relation
\begin{equation}\label{eq:proto-u-relation}
             u_{e}^{3} + w_{1,x} u_{e}^{2} + w_{2,x} u_{e} + w_{3,x} = 0.
\end{equation}
\end{lemma}

\begin{proof}
    Let $s_{1}$, $s_{2}$, $s_{3}$ be distinct basepoints on
    $[0,1]\times e$, let $Z_{1}$, $Z_{2}$, $Z_{3}$ be three disjoint
    sub-bifolds of $\check X$ containing these, and let $U_{j}$ be a
    subvariety dual to $\mathsf{u}_{s_{j}}$ in $\bonf^{*}(Z_{j} ;
    \nu\cap Z_{j})$. A standard gluing argument \cite{KM-book}, used to treat
    composite cobordisms in general, shows that the composite operator
    $u_{e}^{3}$ can be computed from the chain map defined by the
    moduli spaces
\[
                M_{3}(\alpha,\beta) \cap U_{1} \cap U_{2} \cap U_{3}.
\]
   Using further disjoint subsets $Z_{4}$, $Z_{5}$, $Z_{6}$ and three
   distinct basepoints in $X^{o}$, we
   construct dual representatives for the classes $\mathsf{w}_{1}$,
   $\mathsf{w}_{2}$,  $\mathsf{w}_{3}$. Let $Z$ be a larger subset of
   $\check X$ that contains all the $Z_{i}$ but still meets the
   singular set of $\check X$ only in the codimension-2 faces. The
   operator on the left-hand side of \eqref{eq:proto-u-relation} is
   then computed from the intersection
\[
             M_{3}(\alpha,\beta)\cap V
\]
  where $V$ has codimension $3$ and is dual to the zero class in
  $\bonf^{*}(Z; \nu)$. We can therefore construct a codimension-2 stratified
  subvariety $W$ with $\partial W = V$. Let $H$ be the operator
  defined on chains by the intersections $M_{2}(\alpha,\beta)\cap W$. 
   At the chain level, we then
  obtain a chain homotopy formula of the shape
\[
  u_{e}^{3} + w_{1,x} u_{e}^{2} + w_{2,x} u_{e} + w_{3,x} = \partial H
  + H \partial
\]
by counting the ends of the 1-dimensional moduli spaces $M_{3}(\alpha,
\beta) \cap W$. Because no moduli spaces of dimension 4 or more are
involved, there is no bubbling, and the ends of the 1-dimensional
moduli space are all of the form $M_{3}(\alpha,\beta)\cap V$ (giving
the left-hand side) or arise from simple broken trajectories (giving
the right-hand side).
\end{proof}

\subsection{Calculating the Stiefel-Whitney operators}

To complete the proof of Proposition~\ref{prop:u-relation} from
Lemma~\ref{lem:proto-u-relation}, need to compute the operators
$w_{i,x}$.

\begin{proposition}\label{prop:w2}
    On $\Jsharp(\check Y;\Gamma)$, the operators $w_{1,x}$ and
    $w_{3,x}$ are both zero, while $w_{2,x}$ is multiplication by the
    element $p\in R$.
\end{proposition}

\begin{proof}
     We first observe that it is sufficient to prove this in the case
     that $\check Y = S^{3}$. From the definition of $\Jsharp$, we
     recall that
\[
              \Jsharp(S^{3} ; \Gamma) = J( (S^{3},\theta);
              \mu_{\theta}; \Gamma)
\]
     and that this is a copy of $R$. To understand why the special
     case is sufficient, consider the disjoint union $\check Y = \check Y_{1}
     \cup \check Y_{2}$ where 
      \[
      \begin{aligned}
          \check Y_{1} &= \check Y \csum (S^{3}, \theta) \\
          \check Y_{2} &= S^{3} \csum (S^{3}, \theta) 
      \end{aligned}   
    \]
    with marking data $\mu_{\theta}$ implied in both cases. We have
    local systems $\Gamma_{1}$ and $\Gamma_{2}$ from the two copies of
    $\theta$, and we form local system $\Gamma = \Gamma_{1}\otimes_{R}
    \Gamma_{2}$ on $\bonf(\check Y)$.  Because
    the second is a free $R$-module, the instanton homology is a
    tensor product:
   \[
                         J(\check Y; \Gamma ) =  J(\check Y_{1};
                         \Gamma_{1} ) \otimes
                                      J(\check Y_{2}; \Gamma_{2} ).
    \] 
    We can take base points $x_{1}$, $x_{2}$ in the non-singular parts
    of either component. By naturality of the isomorphism, the
    resulting operators on the tensor product are $w_{i,x_{1}} \otimes
      1$ and $1\otimes w_{i,x_{2}}$. By an application of excision and
      its naturality, these two operators are equal. Because $J(\check
      Y_{2}; \Gamma_{2} ) = R$, the operator $w_{i,x_{2}}$ is
      multiplication by an element of $R$. The same therefore holds
      for $w_{i,x_{1}}$. 

     So let us consider the operators
     \[
                w_{i,x}: J((S^{3}, \theta) ; \mu_{\theta}; \Gamma) \to  J((S^{3}, \theta) ; \mu_{\theta}; \Gamma),
     \]
     or in the short-hand notation of webs, the operators
    \[
              w_{i,x} : \Jsharp(\emptyset;\Gamma) \to  \Jsharp(\emptyset;\Gamma) .
    \]
    By Proposition 8.11 of \cite{KM-jsharp}, this instanton homology
    group is $\Z/2$ graded. Being a free rank-1 $R$-module, it is
    non-zero in only one of the two gradings. The operators $w_{1,x}$
    and $w_{3,x}$ have odd degree and are therefore zero.

   It remains to calculate $w_{2,x}$ on
   $\Jsharp(\emptyset;\Gamma)$. In this Morse homology, there is a
   unique non-degenerate critical point $\alpha$ (without the need for
   holonomy perturbation), and to calculate $w_{2,x}$ we need to
   describe the 2-dimensional moduli space $M_{2}(\alpha,\alpha)$ of
   marked anti-self-dual orbifold connections on the cylinder. As in
   section~\ref{sec:maps-to-S1}, the three-dimensional orbifold can be
   described as the quotient of a round sphere $\hat S^{3}$ by the
   $V_{4} = Q_{8}/(\pm 1)$, and marked connections are
   $Q_{8}$-equivariant $\SU(2)$ connections on $\hat S^{3}$. A
   conformal compactification of the cylinder $\R\times \hat S^{3} /
   V_{4}$ is the orbifold $4$-sphere $\hat S^{4} / V^{4}$. The
   fixed-point set of $V_{4}$ on $\hat S^{4}$ is a circle $\hat S^{1}
   \subset \hat S^{4}$. This circle can be identified with the union of the
   lines $\R\times \hat s_{+}$ and $\R\times \hat s_{-}$ in the
   cylinder, together with two points at infinity. 

   The two-dimensional moduli space has action $1/4$ on the orbifold
   cylinder, and its pull-back to the round $\hat S^{4}$ therefore has
   action $1$. In this way, we have identified $M_{2}(\alpha,\alpha)$
   as the space of $Q_{8}$-equivariant $1$-instantons in an $\SU(2)$
   bundle on $\hat S^{4}$. The $1$-instanton moduli space on the round
   4-sphere is a 5-ball, and the subspace of instantons that are
   invariant under the action of $V_{4}$ is an open 2-ball inside the
   5-ball (the one that spans the circle $\hat S^{1}\subset \hat
   S^{4}$). The Uhlenbeck compactification of this $V_{4}$-invariant
   moduli space is the closed $2$-disk, obtained by attaching $\hat
   S^{1}$. 

   Each $V_{4}$-invariant $1$-instanton becomes a $Q_{8}$-equivariant
   by lifting the action to the total space of the
   $\SU(2)$-bundle. However, the lift is not unique. Given one lift,
   we can obtain others by multiplying by any  character
   $Q_{8} \to \{\pm 1\}$. Since $Q_{8}$ has four characters, the
   moduli space $M_{2}(\alpha,\alpha)$ consists of four open disks,
  \[
              M_{2}(\alpha,\alpha) = \bigcup_{i=0}^{3} \mathcal{D}^{2}_{i},
  \]
   each of which has $\hat S^{1}$ as its Uhlenbeck boundary.

   Let $\bar{M}$ be the ``small'' compactification of the equivariant
   moduli space on $\hat S^{4}$, obtained by collapsing $\hat S^{1}$
   to a point. This is a bouquet of four $2$-spheres,
   \[
         \bar M = \bigvee_{0}^{3} \mathcal{S}^{2}_{i}.
    \] 
   If $x$ is
   basepoint in $\hat S^{4} / V_{4}$ which does not lie on $\hat S^{1}$, then the
   basepoint $\SO(3)$-bundle $\mathbb{E}_{x}$ on the moduli space extends as a
   bundle on $\bar{M}$.  

   \begin{lemma}\label{lem:w2-nonzero}
       The Stiefel-Whitney class of the base-point bundle is non-zero
       on each sphere in the bouquet:
          \[
              w_{2}(\mathbb{E}_{x})[ \mathcal{S}^{2}_{i } ] =1, \qquad i=0,1,2,3.
          \]
   \end{lemma}

  We postpone the proof of this lemma until the end of this
  subsection, presenting it after the proof of
  Lemma~\ref{lem:instanton-degree} below.

   To continue the proof of Proposition~\ref{prop:w2} we can choose a
   subset $Z\subset \hat{S}^{4}/ V_{4}$ disjoint from the singular set
   of the orbifold, and we can choose therein a dual representative
   $W_{2}$ for $w_{2}(\mathbb{E}_{x})$ such that the
   $\bar M \cap W_{2}$ is disjoint from the vertex of the bouquet and
   meets each of the spheres $\mathcal{S}^{2}_{i}$ transversely. The
   lemma tells us that $\mathcal{S}^{2}_{i} \cap W_{2}$ is an odd
   number of points. Returning to the viewpoint of the cylinder
   $\R\times (\hat S^{3}/V_{4})$, we learn that the transverse
   intersection $M_{2}(\alpha,\alpha)\cap W_{2}$ consists of an odd
   number of points in each of the disks $\mathcal{D}^{2}_{i}$, for
   $i=0,1,2,3$. So if $\zeta_{i}$ is a single point of
   $\mathcal{D}_{i}$ for each $i$, then we have that $w_{2,x}$ is
   multiplication by the element
   \[
             \Gamma_{\zeta_{0}} +   \Gamma_{\zeta_{1}} +
             \Gamma_{\zeta_{2}} + 
  \Gamma_{\zeta_{3}}  \in \Hom(\Gamma_{\alpha},\Gamma_{\alpha}) = R.
   \]
   To complete the proof of Proposition~\ref{prop:w2} we must compute
   each $\Gamma_{\zeta_{i}}$ and show that the above sum is $P$. The
   following lemma therefore finishes the argument.
\end{proof}

   \begin{lemma}\label{lem:instanton-degree}
       Let $\zeta_{i} \in M_{2}(\alpha,\alpha)$ be a point of
       $\mathcal{D}^{2}_{i}$, for $i=0,1,2,3$. Then, with suitable
       conventions and numbering of the components, we have:
         \[
         \begin{aligned}
             \Gamma_{\zeta_{0}} &= T_{1} T_{2} T_{3} \\
             \Gamma_{\zeta_{1}} &= T_{1} T_{2}^{-1} T_{3}^{-1} \\
             \Gamma_{\zeta_{2}} &= T_{1}^{-1} T_{2} T_{3}^{-1} \\
             \Gamma_{\zeta_{3}} &= T_{1}^{-1} T_{2}^{-1} T_{3} .
         \end{aligned}
          \]
   \end{lemma}

   \begin{proof}[Proof of Lemma~\ref{lem:instanton-degree}]
       Consider the case of $\zeta_{0} \in \mathcal{D}^{2}_{0}$. We
       take this point to be the center of the disk, which is the
       standard $1$-instanton on $\hat S^{4}$ with $\SO(5)$ symmetry. Each of
       the disks corresponds to one choice of how to lift the action
       of $V_{4}$ on the standard $\SO(3)$ 1-instanton to an action of
       $Q_{8}$ on the $\SU(2)$ bundle. We take $\mathcal{D}^{2}_{0}$
       to be the one obtained by identifying the $\SU(2)$ bundle of
       the centered and scaled 1-instanton with the spin bundle $S^{-}$ on $\hat S^{4}$
       and making $Q_{8} \subset \Spin(5)$ act on the spin bundle in
       the standard way. We use the formulae
       \eqref{eq:Gamma-explained} and \eqref{eq:Chern-Weil-h} to
       compute $\Gamma_{\zeta_{0}}$.

       After conformally compactifying $\R\times  \hat S^{3}$ to get
       $\hat S^{4}$, the space $\R\times \hat\gamma_{m}$ becomes one
       hemisphere in the $2$-sphere $\hat S^{2}_{m} \subset \hat
       S^{4}$ which is the fixed point set of $I_{m} \in V_{4}$. The
       $\SO(3)$ instanton bundle on $\hat S^{4}$ is the rank-3 bundle
       $\Lambda^{-}\subset \Lambda^{2}$. Along $\hat S^{2}_{m}$ this
       bundle decomposes under the action of $I_{m}$ as
       \[
              \Lambda^{-}|_{\hat S^{2}_{m}} = \R \oplus K_{m},
       \]
       where the circle bundle $K_{m} \to  \hat S^{2}_{m}$ can be
       identified with the tangent bundle, as a bundle with
       connection. Because we are integrating over half of the sphere,
       we see from \eqref{eq:Chern-Weil-h} that
       \[
       \begin{aligned}
           \Delta \tilde h_{m}(\zeta_{0}) &= \frac{1}{2} \deg K_{m} \\
                                         &= \frac{1}{2} e(
                                         S^{2}_{m}) \\
                                         &= 1.
       \end{aligned}
       \]
       (There is some sign ambiguity remaining in our construction of
       $\Gamma$, but the final sign here fixes our conventions.) From
       \eqref{eq:Gamma-explained}, we therefore obtain
       \[
           \Gamma_{\zeta_{0}} = T_{1} T_{2} T_{3}.
        \]
  
       The other three disks $\mathcal{D}^{2}_{i}$ are obtained by
       changing the lift of the $V_{4}$ action to $Q_{8}$ by a
       non-trivial character of $Q_{8}$. The each non-trivial
       character changes two of the three lifts $\hat I_{1}$, $\hat
       I_{2}$, $\hat I_{3}$ by $-1$. Changing $\hat I_{m}$ by $-1$
       changes the identification $C(I_{m})_{1} \to S^{1}$ by
       $z\mapsto -z$, and so changes $\tilde h_{m}$ to $-\tilde
       h_{m}$. Thus $\Gamma_{\zeta_{i}}$ differs from
       $\Gamma_{\zeta_{0}}$ by changing the sign of the exponent of
       $T_{m}$, for two of the three values of $m$. This proves the lemma.
   \end{proof}

   \begin{proof}[Proof of Lemma~\ref{lem:w2-nonzero}]
       We return to the postponed proof of
       Lemma~\ref{lem:w2-nonzero}. The sphere $\mathcal{S}^{2}_{i}$ is
       obtained from a the closed disk $\bar{\mathcal{D}}^{2}_{i}$ by
       collapsing the boundary to a point. We need to describe both
       the bundle $\mathbb{E}_{x}\to \bar{\mathcal{D}}^{2}_{i}$ and
       the trivialization that is used when collapsing the boundary.
 
       Let $\hat x$ be a lift of $x$ to the round sphere $\hat
       S^{4}$. There is a unique $2$-disk $\hat D^{2} \subset \hat
       S^{4}$ which has boundary $\hat S^{1}$ and contains $\hat
       x$. We may take it that $\hat x$ is the center of the disk. We
       identify the centered 1-instanton bundle with $\Lambda^{-}$ as
       in the previous lemma. The disk $\mathcal{D}_{i}^{2}$ in the
       moduli space $M$ is obtained by applying conformal
       transformations, and we can use these to identify
       $\mathcal{D}_{i}^{2}$ with $\hat D^{2}$ in such a way that the
       basepoint bundle \[ \mathbb{E}_{x} \to \mathcal{D}^{2}_{i} \]
       becomes the bundle \[ \iota^{*}(\Lambda^{-})|_{\hat D^{2}}, \]
       where $\iota$ is the involution on the disk that fixes the
       center $x$. The trivialization of $\Lambda^{-}$ that we must
       use on the boundary of $\hat D^{2}$ is the $V_{4}$-invariant
       trivialization of $\Lambda^{-}$ on the circle $\hat
       S^{1}$. This is the same trivialization as is obtained by
       parallel transport around $\hat S^{1}$.

       The question of whether $w_{2}$ is zero or not now becomes the
       question of whether the trivialization of $\Lambda^{-}$ on
       $\hat S^{1}$ obtained by parallel transport can be extended to a
       trivialization on a disk in $\hat S^{4}$ spanned by $\hat
       S^{1}$. The answer is that it cannot, and this is essentially
       the same point as occurs in the proof of the previous lemma: on
       a suitable disk, we can reduce the structure group of
       $\Lambda^{-}$ to $\SO(2)$, and the $\SO(2)$ bundle has degree
       $1$ with respect to the trivialization.
   \end{proof}

\subsection{Three-edge relations}

Before moving on, we consider a different relation involving the
operators $u_{e}$, that can be proved in the same manner. Let $\check{S}$ be
an orbifold 2-sphere with three singular points with $\Z/2$. Suppose
$S$ is embedded in $\check Y$ as a sub-orbifold. (In terms of the web
$K$ in the three-manifold $Y= | \check Y|$, this is a sphere meeting
$K$ transversely in three points belonging to edges of $K$.)
Corresponding to  the three points of intersection, we have three
operators
\[
     u_{1} , u_{2}, u_{3} : \Jsharp(\check Y; \Gamma) \to \Jsharp(\check Y; \Gamma) .
\]

\begin{lemma}\label{lem:three-edge-relations}
    The three operators $u_{1}$, $u_{2}$, $u_{3}$ satisfy the
    relations
   \[
   \begin{gathered}
       u_{1} + u_{2} + u_{3} = 0 \\
              u_{1} u_{2} u_{3} = 0. 
   \end{gathered}
    \]
\end{lemma}

\begin{proof}
    Let $\check X = [0,1]\times \check Y$, and $Z\subset \check X$ be
    a regular neighborhood of the sphere $\{1/2\} \times S$, meeting
    the singular set only in the codimension-2 strata. On
    $\bonf^{*}(Z)$ we have three cohomology classes $\mathsf{u}_{i}$,
    for $i=1,2,3$, and the first relation in the lemma will follow, just as in the proof of
    Lemma~\ref{eq:proto-u-relation}, if we establish a relation in the
    cohomology of $\bonf^{*}(Z)$,
   \[
            \mathsf{u}_{1} + \mathsf{u}_{2} +\mathsf{u}_{3} = 0.
   \]
   This is equivalent to showing that for any bifold $\SO(3)$
   connection on $S^{1}\times \check S$, the product of the three real
   line bundles $\mathbb{L}_{i} \to S^{1}$ obtained from the three
   orbifold points of $\check S$, is trivial. The triviality of
   $\mathbb{L}_{i}$ is equivalent to the bifold bundle $E$ having
   $w_{2}=0$ on the torus $S^{1}\times \delta_{i}$, where
   $\delta_{i}\subset \check S$ is a circle linking the $i$'th
   singular point. Since the sum of the three tori bounds in the
   smooth part of $S^{1}\times \check S$, the above relation follows.

   The second relation is an algebraic consequence of the first
   relation and the relations $u_{i}^{3} + P u_{i}=0$.
\end{proof}

\section{Proof of the main theorem}

In this section, we present the proof of
Theorem~\ref{thm:deformed-is-Tait}. Although the final statement of
the theorem involves webs $K$ in $\R^{3}$ or $S^{3}$, we continue to
treat the case of a general bifold $\check Y$ correpsonding to a web
$K\subset Y$, for as long as possible. Nevertheless, the 3-manifold
$Y$ will be omitted from our notation, and we write
$\Jsharp(K;\Gamma)$ for $\Jsharp(\check Y; \Gamma)$. We will make
clear that $Y=S^{3}$ in the statements that require it.

\subsection{The edge decomposition}

Let $R'$ be obtained from $R$ by adjoining an inverse $1/P$ for the
element for $P$ \eqref{eq:p}. For any bifold $\check Y$, let $\Gamma'$
be the local system $\Gamma \otimes_{R} R'$. The corresponding
instanton homology group $\Jsharp(K; \Gamma')$ is an
$R'$-module. The polynomial 
\[
           u^{3} + Pu
\]
which annihilates the edge-operators $u_{e}$ is the product of
factors $u$ and $u^{2} + P$ which are coprime in $R'[u]$. That is, we
have
\[
           a u + b (u^{2}+P) =0
\]
in $R'[u]$, where $a=u/P$ and $b=1/P$. It follows that the module
$\Jsharp(K; \Gamma')$ has a direct sum decomposition,
\[
\begin{aligned}
     \Jsharp(K; \Gamma') &= \ker(u_{e}) \oplus \ker
     (u_{e}^{2} + P) \\
       &= \ker (u_{e}) \oplus \im (u_{e}).
\end{aligned}
\]
This is the (generalized) eigenspace decomposition: if we
were to adjoin a square-root of $P$, then these summands become the
eigenspace and generalized eigenspace for the two eigenvalues $0$ and $P^{1/2}$.
There is one such decomposition for each edge $e$ of $K$, and since
the operators belonging to different edges commute, we obtain a
decomposition of the $R'$-module into simultaneous generalized
eigenspaces. To introduce notation for this, let us write
\[
            \Edges(K) = \{\text{edges of $K$}\},
\] 
and given a subset $s\subset
\Edges(K)$, let us write
\[
                    V(K;s) = \left(\bigcap_{e\in s} \ker(u_{e})\right) 
                                \cap \left(\bigcap_{e\notin s} \im(u_{e})\right) .
\]
Then we have a decomposition of the
$R'$-module into  $2^{\#\Edges(K)}$ direct summands:
\begin{equation}
    \label{eq:espace-decomp}
    \Jsharp(K; \Gamma') = \bigoplus_{s\subset \Edges(K)} V(K;s).
\end{equation}

\begin{definition}\label{def:espace-decomp}
    For any web $K\subset Y$, we refer to the above direct sum decomposition of the $R'$-module
    $\Jsharp(K;\Gamma')$ as the \emph{edge decomposition}.
\end{definition}

\subsection{The case of the unknot}

We calculate the edge decomposition for the unknot $K$ in $\R^{3}$,
which has a single, circular edge $e$.

\begin{proposition}\label{prop:u-matrix-U}
    For the unknot $K=U$, the $R$-module $\Jsharp(U ; \Gamma)$ is free of
    rank $3$, and with respect to a standard basis the matrix of the
    operator $u_{e}$ is
\[
             \begin{pmatrix}
               0 & 0 & 0 \\
               1 & 0 & P \\
               0 & 1 & 0
             \end{pmatrix}.
\]
     Over the ring $R'$ the module $\Jsharp(U;\Gamma')$ decomposes as
     a direct sum of $\ker(u_{e})$ and $\im(u_{e})$, which have ranks
     $1$ and $2$ respectively.
\end{proposition}

As in \cite{KM-jsharp}, we obtain information about the unknot by
considering the unknotted 2-sphere $S\subset \R^{4}$ as a foam. We
write $S(m)$ for the 2-sphere with $m$ dots, considered as a cobordism
from the empty web to itself. As a cobordism, it has an evaluation in
$\Jsharp(\, \text{---}\, ; \Gamma)$, which we write as
\[
                 \left\langle S(m) \right\rangle \in R.
\]

\begin{lemma}
    For $m=0,1$ and $2$, we have $\left\langle S(m) \right\rangle = 0,0,1$
    respectively. For $m>2$, we have $\left\langle S(m) \right\rangle = P \left\langle
    S(m-2) \right\rangle$.
\end{lemma}

\begin{proof}
    This follows \cite{KM-jsharp}. For $m<2$, the formula for the
    dimension of the moduli spaces tells us that the zero-dimensional
    moduli spaces have negative action, and are therefore empty. For
    $m=2$, we are evaluating $u_{e}^{2}$ on a compact moduli space of
    flat connections equal to $\RP^{2}$, and the evaluation is
    $1$. The claim for $m>2$ is a consequence of the relation \eqref{eq:u-relation}.
\end{proof}

\begin{proof}[Proof of Proposition~\ref{prop:u-matrix-U}]
Let us write $S$ as the union of two disks $D^{+}$ and $D^{-}$, viewed
as cobordisms from the empty web to $U$ and back. Let $D^{\pm}(m)$ be
the same foams, but with $m$ dots. Let us write
\[
\begin{aligned}
    \mathbf{v}_{m} &= \Jsharp(D^{+}(m) ; \Gamma) \\ &\in
    \Jsharp(U;\Gamma) 
\end{aligned}
\]
and
\[
\begin{aligned}
    \mathbf{a}_{m} &= \Jsharp(D^{-}(m) ; \Gamma) \\ &\in \Hom (
    \Jsharp(U;\Gamma), R).
\end{aligned}
\]

We have $\mathbf{a}_{m} ( \mathbf{v}_{n}) = \left\langle S(m+n) \right\rangle$,
so from the lemma we can read off these pairings. If we set
\[
\begin{aligned}
    \mathbf{b}_{0} &= \mathbf{a}_{2} + P \mathbf{a}_{0} \\
    \mathbf{b}_{1} &= \mathbf{a}_{1}\\
    \mathbf{b}_{2} &= \mathbf{a}_{0}
\end{aligned}
\]
then we can read off that \[ \mathbf{b}_{m}(\mathbf{v}_{n}) =
\delta_{mn}\] for $m, n \le 2$. 

The representation variety of $U$ is
$\RP^{2}$, and there is a holonomy perturbation with three critical
points. The complex that computes $\Jsharp(U;\Gamma)$ therefore has
three generators, so $\Jsharp(U;\Gamma)$ is either free of rank $3$,
or has rank strictly less. But the above relation shows that
$\mathbf{v}_{0}$, $\mathbf{v}_{1}$, $\mathbf{v}_{2}$ generate a free
submodule. So $\Jsharp(U;\Gamma)$ has rank $3$, with the
$\mathbf{v}$'s as a basis. 

To compute the matrix of $u_{e}$ in the basis $\mathbf{v}_{0}$,
$\mathbf{v}_{1}$,  $\mathbf{v}_{2}$, we must compute
$\mathbf{b}_{n}(u_{e} \mathbf{v}_{m})$, which again can be interpreted
as evaluations of 2-spheres with dots:
\[
\begin{aligned}
    \mathbf{b}_{0}(u_{e} \mathbf{v}_{m}) &= \left\langle S(m+3) \right\rangle + P \left\langle S(m+1) \right\rangle \\
    \mathbf{b}_{1}(u_{e} \mathbf{v}_{m}) &= \left\langle S(m+2) \right\rangle\\
    \mathbf{b}_{2}(u_{e} \mathbf{v}_{m}) &= \left\langle S(m+1) \right\rangle.
\end{aligned}
\]
From this one may compute the matrix shown in the proposition. The
kernel of $u_{e}$ is the free rank-1 module spanned by the element
\[
          \begin{pmatrix}
                    P \\ 0 \\ 1
          \end{pmatrix}.
\]
The image of $u_{e}$ is the free rank-2 module spanned by the elements
\[
          \begin{pmatrix}
                    0 \\ 1 \\ 0
          \end{pmatrix}, \quad
      \begin{pmatrix}
                    0 \\0 \\ 1
          \end{pmatrix}.
\]
The same applies after tensoring with $R'$ to make $P$ invertible, and
in that case the free module of rank $3$ is the direct sum of these
two submodules.
\end{proof}

In terms of the notation of \eqref{eq:espace-decomp} and
Definition~\ref{def:espace-decomp}, the conclusion of the proposition
is that 
\begin{equation}
    \begin{aligned}
        \label{eq:V-for-U}
        &\rank V(U;\{e\}) = 1 \\
        &\rank V(U; \;\;\emptyset\;) = 2.
    \end{aligned}
\end{equation}

\subsection{The case of the theta graph}

Next, we calculate the edge decomposition for the theta graph
$\theta$, with its three edges $e_{1}$, $e_{2}$, $e_{3}$.

\begin{proposition}\label{prop:V-for-theta}
    For the theta graph, the deformed instanton homology
    $\Jsharp(\theta; \Gamma)$ is free of rank $6$. In the
    edge-decomposition of the $R'$-module $\Jsharp(\theta; \Gamma')$,
    the non-zero summands are the summands $V(\theta; s)$ where $s
    \subset \{e_{1}, e_{2}, e_{3}\}$ is a singleton. Each of the three
    non-zero summands has rank $2$.
\end{proposition}

Following \cite{KM-jsharp} again, we look at the closed foam
$\Theta\subset \R^{4}$ consisting of three disks with a common circle
as boundary. Let $\Theta (m_{1}, m_{2}, m_{3})$ denote this foam with
$m_{i}$ dots on the $i$'th disk, and let
\[
     \left\langle \Theta (m_{1}, m_{2}, m_{3}) \right\rangle \in R
\]
be the evaluation of the closed foam in $\Jsharp(\, \text{---}\, ;
\Gamma)$. These evaluations are entirely determined by the following lemma.

\begin{lemma}\label{lem:Theta-evals}
    The evaluation $\left\langle \Theta (m_{1}, m_{2}, m_{3}) \right\rangle$ is
    symmetric in the three variables. It is
    zero if all the $m_{i}$ are positive, or if the sum of the $m_{i}$
    is either even or less than three. We have
    \[
        \left\langle \Theta (0,1,2) \right\rangle = 1.
    \]
    Finally, if $m_{1} \ge 3$, then
    \[
         \left\langle \Theta (m_{1}, m_{2}, m_{3}) \right\rangle = P   \left\langle \Theta (m_{1}-2, m_{2}, m_{3}) \right\rangle .
   \]
\end{lemma}

\begin{proof}
    The symmetry is clear. The assertion that the evaluation is zero
    if all the $m_{i}$ are positive follows from the relation
    $u_{1}u_{2}u_{3}=0$ in Lemma~\ref{lem:three-edge-relations}. The
    assertion that the evaluation is zero if the sum of the $m_{i}$ is
    even follows from the dimension formula \cite{KM-jsharp} and the
    fact that bifold bundles for this foam have action a multiple of
    $1/4$. If the sum of the $m_{i}$ is less than three, then the
    moduli spaces that contribute have negative action, so the
    evaluation is zero. The evaluation for $\Theta(0,1,2)$ holds
    because the moduli space of flat connections is the
    three-dimensional flag manifold, as in \cite{KM-jsharp}.
\end{proof}

\begin{proof}[Proof of Proposition~\ref{prop:V-for-theta}]
Mimicking the calculation for the case of the unknot, we introduce the
half-foams $\Theta^{+}$ and $\Theta^{-}$ as cobordims from $\emptyset$
to $\theta$ and from $\theta$ to $\emptyset$, and we  write
\[
\begin{aligned}
    \mathbf{v}(l,m,n) &= \Jsharp(\Theta^{+}(l,m,n) ; \Gamma) \\ &\in
    \Jsharp(\theta;\Gamma) 
\end{aligned}
\]
and
\[
\begin{aligned}
    \mathbf{a}(l,m,n) &= \Jsharp(\Theta^{-}(m) ; \Gamma) \\ &\in \Hom (
    \Jsharp(\theta;\Gamma), R).
\end{aligned}
\]
The pairing between these can be computed from the lemma, and we can
compute that the pairing between the six elements
\begin{equation}\label{eq:theta-basis}
           \mathbf{v}(0,0,0), \mathbf{v}(0,0,1), \mathbf{v}(0,0,2), 
            \mathbf{v}(0,1,0), \mathbf{v}(0,1,1), \mathbf{v}(0,1,2)
\end{equation}
and the similarly-dotted $\mathbf{a}$'s is upper triangular with
entries $1$ on the diagonal. It follows that these six elements
$\mathbf{v}(l,m,n)$ generate a free submodule of rank $6$. On the
other hand, the representation variety of $\theta$ is the flag
manifold, and the Chern-Simons functional has a perturbation with six
critical points. So (much as in the case of the unknot), we conclude
that $\Jsharp(\theta; \Gamma)$ is free of rank $6$ and the elements
\eqref{eq:theta-basis} are a basis.

Four of the six basis elements $\mathbf{v}(l,m,n)$ have $n>0$, from
which it follows that the rank of $\im(u_{3})$ (and hence also the
rank of $\ker(u_{3}^{2} + P)$ is at least four. The same applies to all
the $u_{i}$ by symmetry. The relation $u_{1} + u_{2} +
u_{3}=0$ from Lemma~\ref{lem:three-edge-relations} implies $u_{1}^{2}
+ u_{2}^{2} + u_{3}^{2}=0$, so if
\[
      x \in \ker (u_{1}^{2} + P) \cap \ker (u_{2}^{2} + P) \cap \ker (u_{3}^{2} + P)
\]
then $Px=0$. It follows that
\[
            \im (u_{1}) \cap \im (u_{2}) \cap \im (u_{3})
\]
also consists of $P$-torsion elements, so this submodule has rank $0$.

If we pass now to the free, $R'$-module $\Jsharp(\theta;\Gamma')$, we
see that each submodule $\im (u_{i})$ has rank at least $4$, and that
the intersection of all three is zero. Since the whole module has rank
$6$, it follows that $ \im (u_{2}) \cap \im (u_{3})$ has rank exactly
$2$, and that
\[
\begin{aligned}
    \im (u_{2}) \cap \im (u_{3}) &= \ker (u_{1}) \\
                                 &= V (\theta; \{ e_{1} \}).
\end{aligned}
\]
In the edge-decomposition, the instanton homology
$\Jsharp(\theta;\Gamma')$ is now the direct sum of the three
rank-2 modules  $V (\theta; \{ e_{i} \})$: 
the other summands are zero because all the rank is accounted for.
\end{proof}

There is an additional relation that can be extracted from this
computation.

\begin{lemma}\label{lem:w2-relation}
   For the operators on $\Jsharp(\theta; \Gamma)$, we have the
   relation
   \[
            u_{2} u_{3} + u_{3} u_{1} + u_{1} u_{2} = P.
   \]
\end{lemma}

\begin{proof}
    Since both are free modules, it is sufficient to check this on the
    $R'$-module $\Jsharp(\theta; \Gamma')$. Using the
    edge-decomposition and symmetry, we may reduce to checking it on
    the summand $V(\theta ; \{e_{1}\})$. This rank-2 submodule is
    $\im(u_{2}) \cap \im(u_{3})$ and is spanned by $\mathbf{v}(0,1,1)$
    and $\mathbf{v}(0,1,2)$. Since $u_{1}=0$ on this summand, the
    operator on the left of the relation in the lemma is just
    $u_{2}u_{3}$. So we are left to check that
    \[
              u_{2} u_{3} \mathbf{v}(0,1,n) = P \, \mathbf{v}(0,1,n)
    \]
    for $n=1$ and $n=2$. This is equivalent to 
    \[
                  \mathbf{v}(0,2,n+1) = P \, \mathbf{v}(0,1,n)
    \]
    A basis for the dual of this summand is
    provided by $\mathbf{a}(0,0,0)$ and  $\mathbf{a}(0,0,1)$. So we
    have to show
   \[
              \left\langle \Theta(0,2,m+1) \right\rangle = P
         \left\langle \Theta(0,1,m) \right\rangle 
   \]
   for $m=1,2$ and $3$. Both sides are zero unless $m=2$, in which
   case both sides are $P$, by Lemma~\ref{lem:Theta-evals}.
\end{proof}

\subsection{Edge-decompositions and 1-sets}

The relation of Lemma~\ref{lem:w2-relation}, for operators on the
theta graph, extends to the case of a general web $K\subset Y$, as
follows.  At each vertex, there are three incident edges
$(e_{1}, e_{2}, e_{3})$, where we allow that two of the three may be
the same.  Let $u_{1}$, $u_{2}$, $u_{3}$ be the corresponding operators on the
$R$-module $\Jsharp(K; \Gamma)$. Then we have the following relations.

\begin{proposition}\label{prop:vertex-relations}
    The three operators $u_{1}$, $u_{2}$, $u_{3}$ on  $\Jsharp(K;
    \Gamma)$ corresponding to the edges incident at a vertex satisfy
    the three relations
    \begin{gather}
        u_{1} + u_{2} + u_{3} = 0 \label{eq:vertex-relation-1} \\
          u_{2} u_{3} + u_{3} u_{1} + u_{1} u_{2} = P \label{eq:vertex-relation-2}\\
        u_{1} u_{2} u_{3} = 0.\label{eq:vertex-relation-3}
    \end{gather}
\end{proposition}

\begin{proof}
    The first and third relations are special cases of the relations
    of Lemma~\ref{lem:three-edge-relations}. The middle relation
    \eqref{eq:vertex-relation-2} follows from the case of the theta
    graph (Lemma~\ref{lem:w2-relation}) by an application of excision,
    as in \cite[Proposition 5.8]{KM-jsharp}.
\end{proof}

\begin{remark}
    The three relations of this proposition reflect the fact that, at
    a vertex $x$ of the singular set of the orbifold, we have a direct
    sum decomposition of the base-point bundle
    $\tilde{\mathbb{E}}_{x}$ as a sum of three real line bundles,
    \[
         \tilde{\mathbb{E}}_{x} = \mathbb{L}_{1} \oplus
         \mathbb{L}_{2} \oplus  \mathbb{L}_{3},
   \]
   and corresponding relations among the characteristic classes,
   \[
   \begin{aligned}
       \mathsf{u}_{1} + \mathsf{u}_{3} + \mathsf{u}_{3}
       &=  \mathsf{w}_{1,x} \\
 \mathsf{u}_{2}   \mathsf{u}_{3} + \mathsf{u}_{3}   \mathsf{u}_{1} +  \mathsf{u}_{1}   \mathsf{u}_{2}
       &=  \mathsf{w}_{2,x} \\
        \mathsf{u}_{1}   \mathsf{u}_{2}   \mathsf{u}_{3} &= \mathsf{w}_{3,x}. 
   \end{aligned}
   \]
   One could convert these relations among cohomology classes
   (in particular, the second relation) into
   relations between corresponding operators, using the same ideas as
   the proof of Proposition~\ref{prop:u-relation}. This would provide
   an alternative proof of the operator relations. However, the
   argument needs to be carried out in the space of connections on a
   neighborhood of a vertex, where codimsion-2 bubbling can occur,
   making the proof more difficult.
\end{remark}

 We say that a subset $s\subset E(K)$ is a
\emph{$k$-set} (for $k=1$ or $2$) if, at each vertex of $K$, exactly
$k$ of the three incident edges belongs to $s$. The complement of a
$1$-set is a $2$-set. This is standard
terminology, but since our webs are allowed to have (for example) no
vertices, it is possible for a subset $s$ to be both a $1$-set and a
$2$-set. (In particular, both $\emptyset$ and $\{e\}$ are $1$-sets in
the case of the unknot.) A $2$-set is the same as a disjoint union of
cycles which includes every vertex. A $1$-set is also called a
\emph{perfect matching}. Every bridgeless trivalent graph admits a
perfect matching \cite{Petersen}.

\begin{proposition}
    In the edge-decomposition,
    \[
                \Jsharp(K;\Gamma') = \bigoplus_{s\subset E(K)} V(K ; s),
    \]
   the summand $V(K;s)$ is zero if $s$ is not a $1$-set.
\end{proposition}

\begin{proof}
    Let $u_{1}$, $u_{2}$, $u_{3}$ be the operators corresponding to
    the three edges incident at a vertex of $K$. Over the ring $R'$,
    consider the operators
    \[
    \begin{aligned}
        \pi_{1} &= (1/P) u_{2}u_{3} \\
        \pi_{2} &= (1/P) u_{3}u_{1} \\
        \pi_{3} &= (1/P) u_{1}u_{2} 
    \end{aligned}
    \] 
   on $\Jsharp(K ; \Gamma')$.
   From Proposition~\ref{prop:vertex-relations}, we have
   \[
                   \pi_{1}  + \pi_{2} + \pi_{3} = 1 \\
   \]
   and
   \[
               \pi_{i} \pi_{j} = 0
   \]
   for $i\ne j$. From this it follows that the $\pi_{i}$ are
   projections ($\pi_{i}^{2} = \pi_{i}$) and that their images give a
   direct sum decomposition
   \[
              \Jsharp(K ; \Gamma') = \im(\pi_{1}) \oplus  \im(\pi_{2}) \oplus \im(\pi_{3}).
   \]
   It is also clear that $\im(\pi_{1})$ for example is contained in
   $\im(u_{2}) \cap \im(u_{3})$, and also $\im(\pi_{1}) \subset
   \ker(u_{1})$ by another application of
   \eqref{eq:vertex-relation-3}. So we have
   \begin{equation}\label{eq:pi1-summand}
             \im(\pi_{1}) \subset \ker (u_{1}) \cap \im(u_{2}) \cap \im(u_{3}).
   \end{equation}
   The reverse inclusion holds also, because if $x\in \ker (u_{1})$
   then $\pi_{2}x = \pi_{3}x=0$ from which it follows that $x=\pi_{1}x$.
   We learn that
    \begin{multline}\label{eq:edge-decomp-local}
        \Jsharp(K ; \Gamma') = \ker (u_{1}) \cap \im(u_{2}) \cap
        \im(u_{3}) \\ \oplus \ker (u_{2}) \cap \im(u_{3}) \cap \im(u_{1}) \\ \oplus
        \ker (u_{3}) \cap \im(u_{1}) \cap \im(u_{2}).
    \end{multline}
    But the first term
    on the right is the direct sum of all those summands $V(K;s)$ for which $s$
    contains $e_{1}$ but not $e_{2}$ or $e_{3}$; and the sum of the
    three terms subspaces on the right is the sum of all $V(K;s)$ for
    which $s$ contains exactly on of the three edges.
\end{proof}

\begin{corollary}\label{cor:only-1-sets}
    For any web $K$ in $Y$, we have a direct sum decomposition
     \[
                \Jsharp(K;\Gamma') = \bigoplus_{\text{\normalfont $1$-sets $s$}} V(K ; s).
      \]
\qed
\end{corollary}

\begin{corollary}\label{cor:partial-1-sets}
    For any web $K$ in $Y$ and any subset $t\subset \Edges(K)$, we have
     \[
               \bigcap_{e\in t} \ker(u_{e}) = 0
      \]
    in $\Jsharp(K;\Gamma')$ if $t$ is not contained in a
    $1$-set. Similarly,
    \[
               \bigcap_{e\in t} \im(u_{e}) = 0
      \]
   if $t$ is not contained in a $2$-set.
\qed
\end{corollary}

\subsection{Applications of the exact triangle and excision}

The excision isomorphism which gives rise to the product rule,
Proposition~\ref{prop:tensor-product}, respects the
edge-decompositions of the webs that are involved. As a simple
application, we have the following.

\begin{lemma}\label{lem:0-handle}
Let $K \subset Y$ be a web, and let $s\subset E(K)$ be $1$-set. Let
$K'$ be the union $K\cup U$, where $U$ is an unknot contained in a
ball disjoint from $K$. Let $e$ be the single edge of $U$, and
$s'\subset E(K')$ be the $1$-set $s\cup \{e\}$. Then
\[
         V(K;s) \cong V(K' ; s').
\]
\end{lemma}

\begin{proof}
    Proposition~\ref{prop:u-matrix-U} tells us that $V(U;\{e\})$ is a
    free module of rank $1$. So we apply
    Proposition~\ref{prop:tensor-product} to obtain
    \[
    \begin{aligned}
        V(K'; s') &\cong V(K; s) \otimes V(U; \{e\} ) \\
        &\cong V(K; s).
    \end{aligned}
   \]
\end{proof}

Consider next the exact triangle of Proposition~\ref{prop:exact-triangle}
for the case of $L_{2}$, $K_{1}$, $K_{0}$. The proof works for any
local system, so the exact sequence holds for the $R'$-modules $\Jsharp(\; \text{---}
\; \Gamma')$. For each of $L_{2}$, $K_{1}$ and $K_{0}$, let $p_{1}$,
\dots, $p_{4}$ be points lying at the indicated locations on the webs,
shown in Figure~\ref{fig:skein-edge}.

\begin{figure}
    \begin{center}
        \includegraphics[scale=.5]{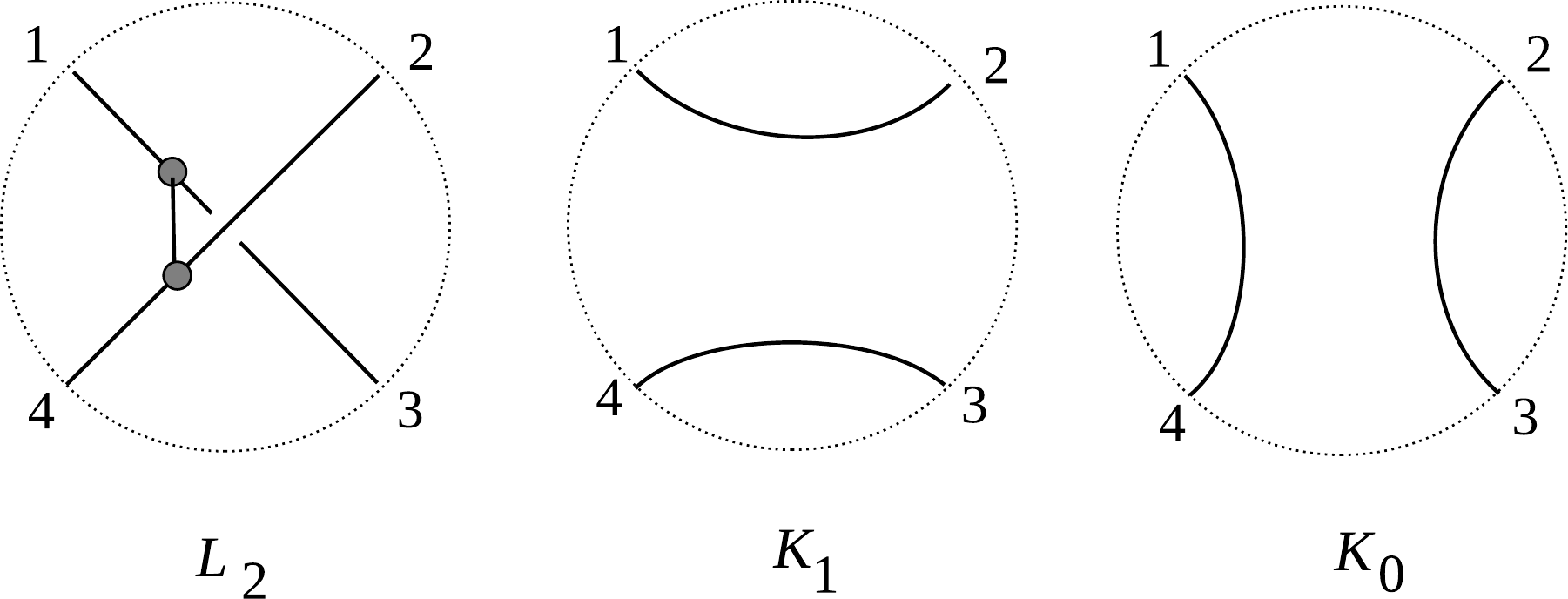}
    \end{center}
    \caption{\label{fig:skein-edge}
    Webs in $Y$ differing inside a ball, with edges labeled.}
\end{figure}

Corresponding to these marked points, we have operators $u_{1}$,
\dots, $u_{4}$ on $\Jsharp(K; \Gamma')$ for each of the three webs,
and the homomorphisms in the exact sequence commute with these
operators, because they come from cobordisms in which the points lie
on common faces. Of course, some of the operators are equal: thus
$u_{1}=u_{2}$, for example, as operators on $\Jsharp(K_{1} ;
\Gamma')$.

For each of $u_{1}$, \dots, $u_{4}$, we have a decomposition of
$\Jsharp(K ; \Gamma')$ into the $\ker(u_{i}) \oplus \im (u_{i})$, and
the long exact sequence of Proposition~\ref{prop:exact-triangle}
decomposes into a direct sum of 16 exact sequences, although many
terms are zero. We give two applications of this idea.

\begin{lemma}\label{lem:iso-1}
    For the webs $K_{1}$ and $K_{2}$, the summands 
    \[
                \bigcap_{i=1}^{4} \ker(u_{i})
    \]
   in $\Jsharp(K_{1}; \Gamma')$ and $\Jsharp(K_{2};\Gamma')$ are isomorphic.
\end{lemma}

\begin{proof}
    The corresponding summand of $\Jsharp(L_{2};\Gamma')$ is zero by Corollary~\ref{cor:partial-1-sets},
    because the edges of $L_{2}$ on which the points $p_{1}$, \dots,
    $p_{4}$ lie are not part of a $1$-set. So the exact sequence for
    these summands becomes an isomorphism between the other two terms.
\end{proof}

\begin{lemma}\label{lem:iso-2}
    For the webs $L_{2}$ and $K_{0}$, the summands 
    \[
                \left( \im(u_{1}) \cap \im(u_{4}) \right) \cap  \left( \ker(u_{2}) \cap \ker(u_{3}) \right)
    \]
   in $\Jsharp(L_{2}; \Gamma')$ and $\Jsharp(K_{0};\Gamma')$ are isomorphic.
\end{lemma}

\begin{proof}
    The corresponding summand of $\Jsharp(K_{1};\Gamma')$ is zero
    because $u_{1}=u_{2}$, so $\ker(u_{1}) \cap \im(u_{2})=0$. So
    again the exact sequence becomes an isomorphism between the other
    two terms.
\end{proof}

We can draw the following corollaries of these two lemmas, in the
language of the edge-decomposition,
Definition~\ref{def:espace-decomp}.

\begin{figure}
    \begin{center}
        \includegraphics[scale=.5]{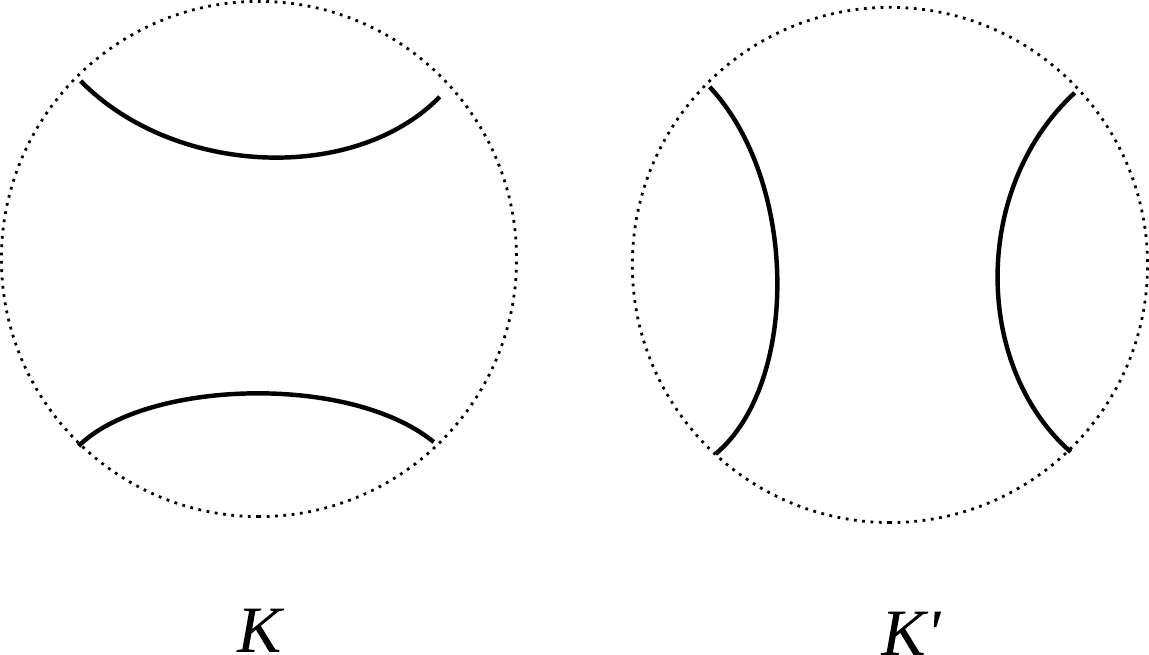}
    \end{center}
    \caption{\label{fig:handle-add}
    Webs $K$ and $K'$ in $Y$ as in Lemma~\ref{cor:handle-add}. The
    edges shown belong to $1$-sets $s$ and $s'$.}
\end{figure}

\begin{corollary}\label{cor:handle-add}
    Let $s$ and $s'$ be $1$-sets for the webs $K$ and $K'$ in
    $Y$. Suppose that $K$ and $K'$ differ only in a ball, as in
    Figure~\ref{fig:handle-add}, and that the edges of $K$ and $K'$ which meet the
    ball belong to $s$ and $s'$ respectively. Then
    \[
                    V(K; s) \cong V(K' ; s').
    \]
\end{corollary}

\begin{proof}
    This is an immediate consequence of the definition of $V(K;s)$ and Lemma~\ref{lem:iso-1}.
\end{proof}

The second corollary is a little more elaborate, but still straightforward.

\begin{figure}
    \begin{center}
        \includegraphics[scale=.5]{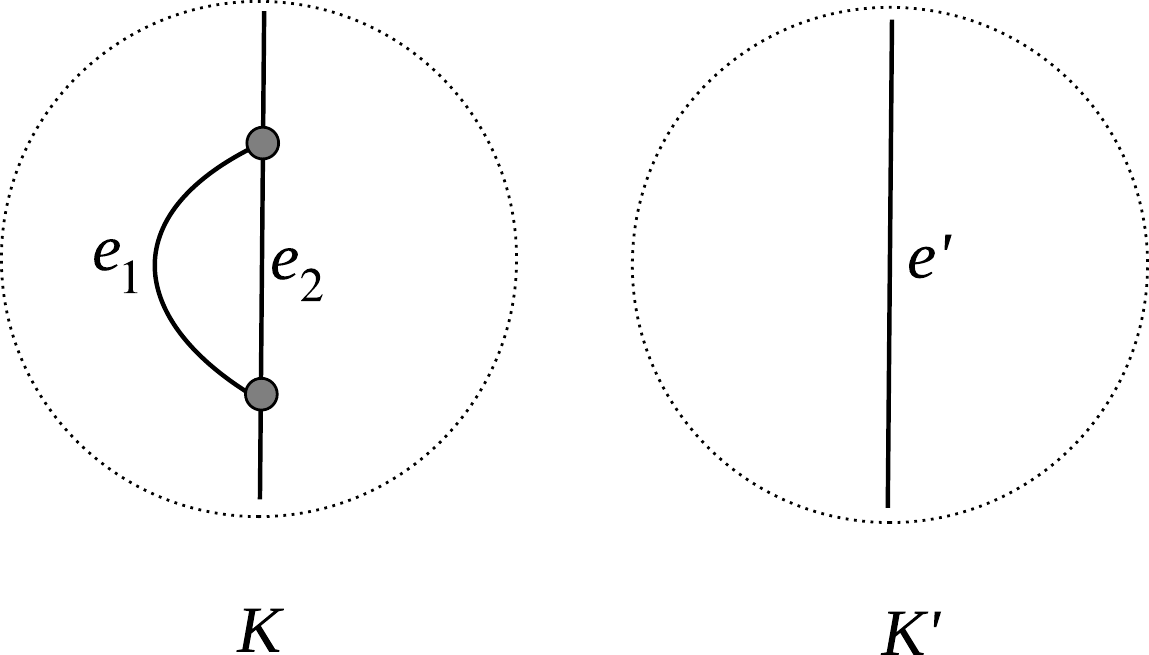}
    \end{center}
    \caption{\label{fig:bigon-add}
    Webs $K$ and $K'$ in $Y$ as in Lemma~\ref{cor:bigon-add}. The
    edge $e_{1}$ belongs to the $1$-set $s$. The other edges do not. }
\end{figure}

\begin{corollary}\label{cor:bigon-add}
    Let $s$ and $s'$ be $1$-sets for the webs $K$ and $K'$ in
    $Y$. Suppose that $K$ and $K'$ differ only in a ball, as in
    Figure~\ref{fig:bigon-add}. Suppose that the edge $e_{1}$ of $K$
    belongs to $s$ (so that the other edges of $K$ the figure do not).
    Suppose that the edge $e'$ of $K'$ does not belong to
    $s'$. Then
    \[
                    V(K; s) \cong V(K' ; s').
    \]
\end{corollary}

\begin{proof}
     Set  $L_{2}=K$ and  apply Lemma~\ref{lem:iso-2} to the ball
    which forms a neighborhood of the edge $e_{2}$. Then the 
    web $K_{0}$ in  Lemma~\ref{lem:iso-2} becomes a union of $K'$ and
    an unknotted circle $U$. Writing $e$ for the single edge of $U$,
    we learn that
    \[
         V(K; s) = V(K' \cup U ; s' \cup \{e\}) .
    \]
    Now apply Lemma~\ref{lem:0-handle}.
\end{proof}

We can summarize Corollary~\ref{cor:handle-add}, Corollary~\ref{cor:bigon-add} and
Lemma~\ref{lem:0-handle} together as saying that $V(K,s)$ is unchanged
if we alter only the edges of $K$ that belong to $s$, by the addition
of 1-handles, 0-handles, or 2-handles (saddle-moves, and births and
deaths of circles).  

\begin{proposition}\label{prop:s-homology}
    Let $(K,s)$ and $(K', s')$ be two webs in $Y$, each equipped with
    $1$-sets $s\subset E(K)$ and $s'\subset E(K')$. Let
    \[
    \begin{aligned}
        C &= \bigcup_{e\in E(K)\setminus s} \bar e \\
        C' &= \bigcup_{e'\in E(K')\setminus s'} \bar e'
    \end{aligned}
     \]
    be the closed loops in $Y$ formed by the edges of the
    complementary $2$-sets. Suppose that $C=C'$ as subsets of $Y$, and
    that the $1$-sets $s$ and $s'$ define the same relative homology
    class in $H_{1}(Y,C; \Z/2)$. Then
    \[
               V(K, s) \cong V(K', s').
    \]
\end{proposition}

\begin{proof}
    The homology between $s$ and $s'$ is a composition of isotopies,
    births, deaths and saddle moves, and the addition and subtraction
    of bigons as in Corollary~\ref{cor:bigon-add}.
\end{proof}

\subsection{Calculation for planar webs}

We now turn to the special case of webs $K$ in the plane
$\R^{2}\subset \R^{3}$. Let $s$ be a $1$-set for $K$, and let $C$ be
the union of closed cycles formed by the complementary $2$-set.  We say that $s$
is \emph{even} if its homology class in $H_{1}(\R^{3}, C ;\Z/2)$ is
zero. This is equivalent to saying that $s$ has an even number of
endpoints on each of the connected components of $C$.

\begin{proposition}
    Let $K$ be a planar web and let $s$ be an even $1$-set. Then $V(K,
    s)$ is a free $R'$-module of rank $2^{n}$, where $n$ is the number of
    components in $C$. If $s$ is not even, then $V(K,s)=0$. 
\end{proposition}

\begin{proof}
    If $s$ is even, then we can apply
    Proposition~\ref{prop:s-homology} to see that $V(K,s)$ is
    isomorphic to $V(K', \emptyset)$, where $K'$ is the disjoint union
    of the circles that comprise $C$, with no vertices. Since $K$ is
    planar, $K'$ is an unlink. If it has $n$ components, then by the
    product formula in Proposition~\ref{prop:tensor-product} and the
    unknot calculation in Proposition~\ref{prop:u-matrix-U}, we have
    \[
    \begin{aligned}
        V(K', \emptyset) &= V(U, \emptyset)^{\otimes n} \\
        &= (R' \oplus R')^{\otimes n}.
    \end{aligned}
     \]
    This establishes the first claim.

    For the second claim, suppose $s$ is not even, and let $C_{1}
    \subset C$ be a connected component on which $s$ has an odd number
    of endpoints. Then, using Proposition~\ref{prop:s-homology}, we
    can replace $(K,s)$ with $(K', s')$,  where the cycles of the
    complementary $2$-set are unchanged, but $s'$ has exactly one
    endpoint on $C_{1}$, and $V(K,s) = V(K', s')$. Furthermore, we can
    arrange that $C_{1}$ bounds a disk that is disjoint from the rest
    of $K'$. In the language of \cite{KM-jsharp}, the web $K'$ now has
    an embedded bridge: that is, there is a 2-sphere in $\R^{3}$
    meeting $K'$ transversely in a single point. It follows that the
    bifold representation variety of $K'$ is empty and $\Jsharp(K'
    ;\Gamma')=0$. So $V(K', s') = 0$ a fortiori. 
\end{proof}

\begin{corollary}
    For a planar web $K$, the $R'$-module $\Jsharp(K, \Gamma')$ is a
    free module of rank
    \[
                        \sum_{s \in \{\text{\normalfont even $1$-sets}\}} 2^{n(s)}
   \]
   where, for each $s$, the number $n(s)$ is the number of cycles in
   the complementary $2$-set. \qed
\end{corollary}

\begin{corollary}
    For a planar web $K$, the rank of the $R'$-module $\Jsharp(K,
    \Gamma')$ is equal to the number of Tait colorings of $K$.
\end{corollary}

\begin{proof}
    It is an elementary fact that the number of Tait colorings is
    equal to the sum that appears in the previous corollary. Indeed,
    given a
    Tait coloring, the edges colored ``red'' are  a $1$-set, and to
    complete the coloring we must alternate ``blue'' and ``green''
    along the cycles of the complementary $2$-set. This will not be
    possible if the $1$-set is not even, and can be done in $2^{n(s)}$
    ways when it is.
\end{proof}

Finally, we return to the ring $R$ and the local system $\Gamma$. We
do not know whether $\Jsharp(K; \Gamma)$ is a free $R$-module, but we
do know that its rank is the same as that of the $R'$-module
$\Jsharp(K; \Gamma')$. This is because both ranks are equal to the dimension of the
$F$-vector space $\Jsharp(K; \Gamma\otimes_{R} F)$, where $F$ is the
field of fractions of both $R$ and $R'$. This proves
Theorem~\ref{thm:deformed-is-Tait}. \qed

\section{Comparison of the deformed and undeformed homology}

\subsection{The spectral sequence  of the $\fm$-adic filtration} 
\label{sec:spectral-sequence}

Recall that $R$ is the ring of finite Laurent series in variable
$T_{1}$, $T_{2}$, $T_{3}$, and that $\Gamma$ is a local system of free
rank-1 $R$-modules on $\bonf^{\sharp}(\check Y)$. Let $\sm\ideal R$ be
the ideal generated by
\[
    \{ \, T_{i} - 1 \mid i=1,2,3 \,\}, 
\]
let $R_{\sm}$ be the localization of $R$ at this maximal ideal: the
ring of rational functions whose denominator is non-zero at
$(1,1,1)$. Let $\fm \ideal R_{\sm}$ be the unique maximal ideal in the
localization. Let $\Gamma_{m} = \Gamma\otimes_{R} R_{m}$ be the local
system of $R_{m}$-modules obtained from $\Gamma$. 

Using the local system $\Gamma_{m}$ in place of $\Gamma$, we can form
the instanton homology group $\Jsharp(\check Y; \Gamma_{m})$. Its rank
as an $R_{m}$-module is equal to the rank of $\Jsharp(\check Y;
\Gamma_{m})$ as an $R$-module, because $R_{m}$ and $R$ have the same
field of fractions. In particular, if $\check Y = (S^{3}, K)$ and $K$
is planar, then the rank is equal to the number of Tait colorings of
$K$. 

Write
\begin{equation}\label{eq:Csharpm}
    \begin{aligned}
        C(\Gamma_{m}) &= C^{\sharp}(\check Y ; \Gamma_{m}) \\ &=
        \bigoplus_{\beta} \Gamma_{m,\beta}
    \end{aligned}
\end{equation}
for the chain complex whose homology is $\Jsharp(\check Y ;
\Gamma_{m})$, as in \eqref{eq:Csharp}. Although the terminology
``complex'' is traditional, it should be remembered that
$C(\Gamma_{m})$ has no grading in general. It is a
differential $R_{m}$-module. We write $\partial_{m}$ for the
differential, so
\[
\Jsharp(\check Y ; \Gamma_{m}) = H( C(\Gamma_{m}), \partial_{m}).
\]
The rank-1 local system $\Gamma_{m}$ has the $\fm$-adic filtration
\[
            \Gamma_{m} \supset \fm \Gamma_{m} \supset \fm^{2}
            \Gamma_{m} \supset \cdots,
\]
and there is the corresponding filtration of the differential group,
\[
            C(\Gamma_{m}) \supset \fm  C(\Gamma_{m}) \supset \fm^{2}
             C(\Gamma_{m}) \supset \cdots.
\]
By its construction, $\fm^{p} C(\Gamma_{m})$ is the same as $C(\fm^{p}
\Gamma_{m})$. The $\fm$-adic filtration of the differential group
gives rise to an induced filtration of the homology:
\begin{equation}\label{eq:F-filtration}
      \Jsharp(\check Y ; \Gamma_{m}) = \cF^{0}  \supset \cF^{1}  \supset
      \cF^{2} \cdots ,
\end{equation}
where as usual $\cF^{p}$ is the subset of the homology that can be
represented by cycles in $\fm^{p} C(\Gamma_{m})$.

For the statement of the next proposition, 
recall that $\Jsharp(\check Y)$ denotes the instanton homology with
coefficients $\F = \Z/2$.

\begin{proposition}\label{prop:spectral-sequence}
    There is a convergent spectral sequence of differential
    $R_{m}$-modules whose $E_{1}$ page is the filtered module
    \[ 
        \Jsharp(\check Y) \otimes \gr R_{m}
     \]
    with the filtration obtained from the $\fm$-adic filtration of
    $R_{m}$, and which abuts to the filtered module $\Jsharp(\check
    Y ; \Gamma_{m})$ with the filtration $\cF^{p}$ induced by the $\fm$-adic
    filtration of $C(\Gamma_{m})$. Thus,
    \[
           \Jsharp(\check Y) \otimes \gr R_{m} \implies \gr \Jsharp(\check
              Y ; \Gamma_{m}).
    \]
     Furthermore, the filtration $\cF^{p}$ is $\fm$-stable, in that
     $\fm \cF^{p} \subset \cF^{p+1}$ with equality for large enough $p$.
\end{proposition}

\begin{proof}
    Since $C(\Gamma_{m})$ has finite rank as an $R_{m}$-module, the
    existence of a convergent spectral sequence from the filtered
    differential module $C(\Gamma_{m})$ is standard. 
    See for example \cite[Theorem A3.22]{Eisenbud}. Note that, 
      although the situation most-often
    considered is a graded differential module (a complex in the
    traditional sense), we can create a complex from the 
    differential module by placing
    $C(\Gamma_{m})$ in every degree,
     \[
               \cdots \stackrel{\partial_{m}}{\longrightarrow}
               C(\Gamma_{m})
               \stackrel{\partial_{m}}{\longrightarrow} C(\Gamma_{m})
                \stackrel{\partial_{m}}{\longrightarrow} C(\Gamma_{m})
           \stackrel{\partial_{m}}{\longrightarrow} \cdots.
     \]
   In the present example, the $E_{1}$ page is the homology of the
   associated graded differential module, 
   \[
         \left( \fm^{p} C(\Gamma_{m}) \right) / \left( \fm^{p+1}
             C(\Gamma_{m})
             \right) =
                   C\left ( (\fm^{p}(\Gamma_{m}) /(
                      \fm^{p+1}\Gamma_{m}) 
            \right).
\] 
The local system $ (\fm^{p}\Gamma_{m}) /(\fm^{p+1}\Gamma_{m})$ is
isomorphic to the
constant coefficient system $\fm^{p}/\fm^{p+1}$, because each $T_{i}$
acts trivially on the quotient. This identifies the $E_{1}$ page as
$\Jsharp(\check Y) \otimes \gr R_{m}$ as claimed. For the claim of
$\fm$-stability, see  \cite[Exercise A3.42]{Eisenbud}.
\end{proof}

We have the following corollary (which can also be proved simply and
directly, as pointed out in the introduction).

\begin{corollary}\label{cor:rank-inequality}
    There is an inequality of ranks,
    \[
    \begin{aligned}
        \dim_{\,\F} \Jsharp(\check Y) &\ge \rank_{R_{m}} \Jsharp(K;
        \Gamma_{m})\\
        &= \rank_{R} \Jsharp(K; \Gamma).
    \end{aligned}
     \]
\end{corollary}

\begin{proof}
    The spectral sequence tells us that 
     \[
                \dim_{\,\F}\left( \Jsharp(\check Y) \otimes (\fm^{p} /
                    \fm^{p+1}) \right)
                      \ge \dim_{\,\F}\left (\cF^{p} /\cF^{p+1}\right), 
     \]
   for these are the dimensions of the associated graded modules on
   the $E^{1}$ page and the limit respectively. Because the filtration
   is $\fm$-stable, we have
   \[
                \dim_{\,\F}\left (\cF^{p} /\cF^{p+1}\right) \sim
                        \dim_{\,\F} \left ( \fm^{p} \Jsharp(\check
                        Y;\Gamma_{m})  \bigm / \fm^{p+1} \Jsharp(\check
                        Y;\Gamma_{m}) \right),
    \]
    and the right-hand side is asymptotic to
    $n \dim (\fm^{p} /\fm^{p+1})$, where $n$ is the rank of
    $ \Jsharp(\check Y;\Gamma_{m})$.
\end{proof}

\begin{corollary}[Corollary~\ref{cor:Jsharp-planar-inequality} of the Introduction]
    If $K\subset \R^{2}\subset \R^{3}$ is a planar web, then the
    dimension of $\Jsharp(K)$ is greater than or equal to the number
    of Tait colorings of $K$.
\end{corollary}

\begin{proof}
    This is now an immediate consequence of the previous corollary and
    Theorem~\ref{thm:deformed-is-Tait}.
\end{proof}

The simplest example where this provides a new bound is the case of
the 1-skeleton of the dodecahedron. Previous work \cite{KM-jsharp,
  KM-jsharp-triangles} provided a lower bound of $58$ in this case, but the
corollary provides a lower bound of $60$. 
The authors do not know whether the inequality of
Corollary~\ref{cor:Jsharp-planar-inequality} continues to hold if $K$
is not planar.

\subsection{Criteria for equality of ranks}

We examine when equality can occur for the inequality of ranks in
Corollary~\ref{cor:rank-inequality}. For this purpose, we introduce a
variant of the local system $\Gamma$. First, let us extend the ground
field $\F$ by transcendentals $z_{1}$, $z_{2}$, $z_{3}$, and write
\[
        \tilde\F = \F(z_{1}, z_{2}, z_{3}).
\]
We replace our ring $R$ and its localization $R_{m}$ by $\tilde R=
R\otimes\tilde \F$ and its localization $\tilde R_{m}$. The latter is
the local ring at the point $(1,1,1)$ in $\mathbb{A}^{3}$, and we now
wish to restrict to a generic line through $(1,1,1)$, namely a line
\[
       (T_{1}, T_{2}, T_{3}) = (1 + z_{1} t, 1+z_{2} t, 1 + z_{3} t).
\]
Thus we introduce the ring $S$ which is the localization at $0\in
\mathbb{A}^{1}$ of the polynomial ring $\tilde\F[t]$, and we regard
$S$ as a module over $\tilde R$ by
\begin{equation}\label{}
           q_{z}:  T_{i} \mapsto 1 +  z_{i} t.
\end{equation}
We write
\[
          \fn \ideal S
\]
for the maximal ideal in this local ring. We have a local system of
$S$-modules
\[
             \Gamma_{S} = \Gamma\otimes_{R} S
\]
and instanton homology groups $\Jsharp(K ; \Gamma_{S})$. 

The analysis of $\Jsharp(K ; \Gamma_{S})$ runs in the same way as
$\Jsharp(K ; \Gamma)$ and $\Jsharp(K ; \Gamma_{m})$. The main point is
that the image of $P$ under the homomorphism $q_{z}$ is non-zero
element $P_{S} \in S$. It has the form
\begin{equation}\label{eq:PS}
             P_{S} = \left( \sum_{i < j} z_{i}^{2} z_{j}^{2} \right)
             t^{4} + O(t^{5}).
\end{equation}
Our edge operators $u_{e}$ now satisfy $u_{e}^{2} + P_{S} u_{e} = 0$,
and the fact that $P_{S}$ is non-zero is sufficient for us to repeat
the previous arguments. The ranks of $\Jsharp(K;\Gamma_{S})$ and
$\Jsharp(K ; \Gamma)$ as an $S$-module and an $R$-module respectively are the
same, because the field of fractions of $S$ is obtained by adjoining
$t$ to the field of fractions of $R$. This allows us to carry over
Theorem~\ref{thm:deformed-is-Tait} diectly; or we could repeat the
same proof. Either way, we have:

\begin{proposition}
     \label{prop:S-is-Tait}
   If $K$ lies in the plane, then the rank of $\Jsharp(K;\Gamma_{S})$ as
   an $S$-module is equal to the number of Tait colorings of $K$. \qed
\end{proposition}

Consider the $\fn$-adic filtration of the local system $\Gamma_{S}$,
and the corresponding filtration of the differential $S$-module
$C^{\sharp}(\Gamma_{S})$ which computes $\Jsharp(K; \Gamma_{S})$, as
in section~\ref{sec:spectral-sequence}. We have an induced filtration
of $\Jsharp(K; \Gamma_{S})$,
\begin{equation}\label{eq:F-filtration}
      \Jsharp(K ; \Gamma_S) = \cG^{0}  \supset \cG^{1}  \supset
      \cG^{2} \cdots ,
\end{equation}
as the counterpart to filtration~\eqref{eq:F-filtration} of
$\Jsharp(K ; \Gamma_{m})$. We therefore have a spectral sequence, just
as in Proposition~\ref{prop:spectral-sequence}.

\begin{proposition}\label{prop:spectral-sequence-S}
    There is a convergent spectral sequence of ungraded differential
    $S$-modules whose $E_{1}$ page is the filtered module
    \[ 
        \Jsharp(\check Y) \otimes \gr S
     \]
    with the filtration obtained from the $\fn$-adic filtration of
    $S$, and which abuts to the filtered module $\Jsharp(\check
    Y ; \Gamma_{S})$ with the filtration $\cG^{p}$ induced by the $\fn$-adic
    filtration of $C(\Gamma_{S})$. Thus,
    \[
           \Jsharp(\check Y) \otimes \gr S \implies \gr \Jsharp(\check
              Y ; \Gamma_{S}).
    \]
     Furthermore, the filtration $\cG^{p}$ is $\fn$-stable, in that
     $\fn \cG^{p} \subset \cG^{p+1}$ with equality for large enough
     $p$. \qed
\end{proposition}

The reason for replacing the $3$-dimensional regular local ring
$R_{m}$ with the $1$-dimensional local ring $S$ is that the
induced filtration $\cG^{p}$ on the homology is much easier to
understand in the case that $S$ is a principal ideal domain. The above
spectral sequence is now an example of Bockstein spectral sequence,
associated in this case to the exact coefficient sequence,
\[
           0 \to S \stackrel{t}{\to} S \to \tilde \F \to 0.
\]
The next simplification in the case of principal ideal domain is that
the induced filtration on the homology of a complex is easier to understand.

\begin{lemma}
    Let $S$ be a domain, and let
    $(C,\partial)$ be  a differential $S$-module,  with $C$ a
    free $S$-module. Let $H(C)$ be its homology. Let $\fn = (t)\ideal
    S$ be a
    principal ideal generated by  a prime $t$, and let $\cG^{p}$ be the $p$'th
    step of the filtration of $H(C)$ induced by the $\fn$-adic
    filtration of $C$. Then $\cG^{p} = \fn^{p} H(C)$.
\end{lemma}

\begin{proof}
    Membership of $\cG^{p}$ means that a homology class $\alpha$ can
    be written in the form $[a]$ where $\partial a=0$ and $a\in
    \fn^{p} C$. The latter condition means that $a=t^{p} b$ for some
    $b$, and the condition $t^{p}\partial b = 0$ implies $\partial b
    =0$. So there is a homology class $\beta = [b]$ with $\alpha =
    t^{p}\beta$. So $\alpha \in \fn^{p} H(C)$. The reverse inclusion
    is straightforward, whether or not the ideal is principal.
\end{proof}

So the induced filtration of $\Jsharp(\check Y; \Gamma_{S})$ is the
$\fn$-adic filtration of the homology group as an $S$-module, 
and the associated graded object to which the spectral sequence of
Proposition~\ref{prop:spectral-sequence-S} abuts has terms
\[
            \gr_{p} \Jsharp(\check
              Y ; \Gamma_{S}) = \frac{\fn^{p}     \Jsharp(\check Y ;
              \Gamma_{S}) }{\fn^{p+1}     \Jsharp(\check Y ;
              \Gamma_{S}) }
\]
So the spectral sequence in
Proposition~\ref{prop:spectral-sequence-S} implies an inequality of
ranks,
\[
\dim_{\,\F} \Jsharp(\check Y) \ge \dim_{\tilde\F} \left( \frac{\fn^{p}
      \Jsharp(\check Y ; \Gamma_{S}) }{\fn^{p+1} \Jsharp(\check Y ;
      \Gamma_{S}) } \right).
\]
Equality holds if and only if all differentials $d_{r}$ in the
spectral sequence vanish.
By Nakayama's lemma, we have
\[
                    \dim_{\tilde\F} \left( \frac{
      \Jsharp(\check Y ; \Gamma_{S}) }{\fn \Jsharp(\check Y ;
      \Gamma_{S}) } \right) \ge \rank_{S} \Jsharp(\check Y ; \Gamma_{S}),
\]
with equality only if $\Jsharp(\check Y; \Gamma_{S})$ is a free
$S$-module. Further, in the case that $\Jsharp(\check Y; \Gamma_{S})$
is free, all the pieces of the associated grade object have dimension
equal to the rank of the module. This proves the following
proposition.

\begin{proposition}
\label{prop:S-rank-inequality}
 We have
\[
          \dim_{\,\F} \Jsharp(\check Y) \ge \rank_{S} \Jsharp(\check Y; \Gamma_{S}),
\]   
and equality holds if an only if all differentials $d_{r}$, $r\ge 1$,
in the spectral sequence of Proposition~\ref{prop:spectral-sequence-S}
are zero, and in that case the module $\Jsharp(\check Y; \Gamma_{S})$
is free. \qed
\end{proposition}

Although we have obtained a criterion from the Bockstein spectral
sequence, the other tool one can use for in the case of principal
ideal domain is the universal coefficient theorem for homology. Our
differential module $C^{\sharp}(\Gamma_{S})$ is a free $S$-module,
from which $C^{\sharp}(\tilde\F)$ is obtained by reducing mod
$(t)$. The universal coefficient theorem therefore tells us that we
have split short exact sequence and an isomorphism
\[
               \Jsharp(\check Y) \otimes \tilde\F
                     \;  \cong\;\Jsharp(\check Y ; \Gamma_{S}) \otimes
                       \tilde\F \; \oplus \; \mathop{\rm{Tor}} ( 
 \Jsharp(\check Y ; \Gamma_{S}) , S/(t)).
\]
Concretely, the finitely-generated $S$-module $ \Jsharp(\check Y ;
\Gamma_{S})$ has a decomposition
\[
         \Jsharp(\check Y ; \Gamma_{S}) \cong S^{r} \oplus
         \frac{S}{(t^{a_{1}})} 
                    \oplus \dots \oplus  \frac{S}{(t^{a_{l}})},
\]
and the universal coefficient theorem tells us that
\[
          \dim_{\,\F} \Jsharp( \check Y) = r + 2 l.
\]
In particular, equality holds in
Proposition~\ref{prop:S-rank-inequality} if and only if
$\Jsharp(\check Y; \Gamma_{S})$ is torsion free.
As a corollary, combining this with Proposition~\ref{prop:S-is-Tait},
we obtain the following.

\begin{corollary}
    Let $K$ be a planar web. Then $\dim_{\,\F} \Jsharp(K)$ is greater
    than or equal to
    the number of Tait colorings, and equality holds
     if and only if one of the following two equivalent conditions
     holds:
     \begin{enumerate}
     \item the spectral sequence 
    of Proposition~\ref{prop:spectral-sequence-S} collapses at the
    $E_{1}$ page; or
    \item the instanton homology $\Jsharp(K ; \Gamma_{S})$ is torsion free.
     \end{enumerate}
     \qed
\end{corollary}

\remark
While we chose to introduce three indeterminates $z_{i}$ in order to
have a general line through $(1,1,1)$, it is apparent from the
formula \eqref{eq:PS} that $P_{S}$ is non-zero if we make either of
the substitutions $(z_{1}, z_{2}, z_{3}) = (1,1,1)$, or $(z_{1},
  z_{2}, z_{3}) = (1,1,0)$. In place of the ring $S$, we could have
    used the smaller ring $\bar S$ obtained as the localization of
    $\F[t]$ at $t=0$, made into an $R$-module by either of the two
    homomorphisms
\[
\begin{aligned}
    q_{(1,1,1)} : R &\to \bar S \\
    q_{(1,1,0)} : R &\to \bar S
\end{aligned}
\]
given by
\[
       \begin{aligned}
    q_{(1,1,1)} : (T_{1}, T_{2}, T_{3}) &\mapsto (1+t, 1+t, 1+t) \\
    q_{(1,1,0)} : (T_{1}, T_{2}, T_{3}) &\mapsto  (1+t, 1+t, 1) 
\end{aligned}      
\]
respectively.

\subsection{Non-vanishing for the homology with local coefficients}

Returning to the three-dimensional local ring $R_{m}$ or to $R$ itself
and the original local coefficient system $\Gamma$, the universal coefficient
theorem provides not a long exact sequence but a Tor spectral
sequence,  because we are not
dealing with a principal ideal domain. Namely, there is a spectral sequence with
\[
  E^{2}_{p} = \mathrm{Tor}^{R}_{p} ( \Jsharp(\check Y;\Gamma), \F)
\]
abutting to $\Jsharp(\check Y)$. As an $R$-module, $\F$ has a
resolution by a complex of free $R$-modules of length $3$, so the Tor
groups that appear on the $E^{2}$ page are zero for $p> 3$. (So only
the $d_{2}$ and $d_{3}$ differentials are potentially non-zero.) A
useful consequence is that, if $\Jsharp(\check Y; \Gamma)$ is zero,
then so is $\Jsharp(\Gamma)$. From the non-vanishing theorem in
\cite{KM-jsharp} we therefore obtain a non-vanishing theorem for
$\Jsharp(\check Y; \Gamma)$.

\begin{theorem}\label{thm:non-vanishing-Gamma}
    Let $K \subset \R^{3}$ be a web with no spatial bridge. Then
    $\Jsharp( K ; \Gamma)$ is a non-zero $R$-module. \qed
\end{theorem}

\subsection{Two non-planar examples} 
\label{sec:counterexample}

The diagram 
on the left in Figure~\ref{fig:tangled-cuffs} is an example of a
non-planar web $K_{1}\subset \R^{3}$ for
which the ranks of $\Jsharp(K_{1})$ and $\Jsharp(K_{1}; \Gamma)$ are
different. 

\begin{figure}
    \begin{center}
        \includegraphics[scale=.35]{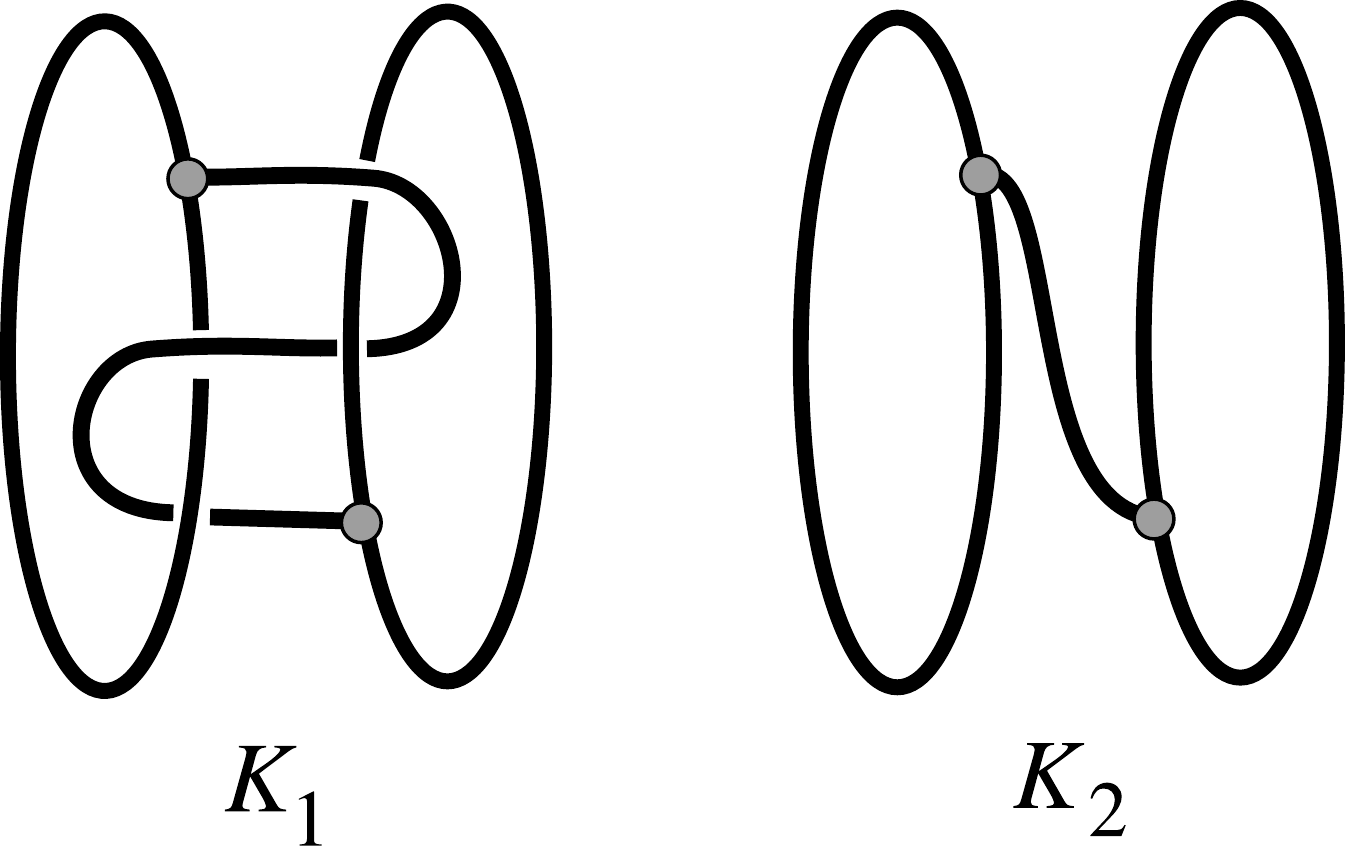}
    \end{center}
    \caption{\label{fig:tangled-cuffs}
   The tangled handcuffs and the standard handcuffs. }
\end{figure}

The web $K_{1}$ admits only one $1$-set, which is the singleton
$\{e_{1}\}$ consisting of the $S$-shaped ``chain'' of the handcuffs. From
the edge-decomposition, we therefore learn that
$\Jsharp(K_{1};\Gamma') = V(K_{1}; \{e_{1}\})$. Similarly, for the
other web in the picture, $\Jsharp(K_{2};\Gamma') = V(K_{2};
\{e_{2}\})$. 
From an application
of Proposition~\ref{prop:s-homology}, we have $V(K_{1}; \{e_{1}\})
=V(K_{2}; \{e_{2}\})$, and it follows that there is an isomorphism
\[
          \Jsharp(K_{1}; \Gamma') \cong \Jsharp(K_{2}; \Gamma').
\]
But the web $K_{2}$ has an embedded bridge, so $ \Jsharp(K_{2};
\Gamma') =0$. It follows that for the
``tangled handcuffs'' $K_{1}$, we have  $\Jsharp(K_{1};
\Gamma')=0$. For the coefficient system $\Gamma$, we do not have a
complete calculation, but we can record at least the following
consequence.

\begin{proposition}
  For the tangled handcuffs $K_{1}$ in Figure~\ref{fig:tangled-cuffs},
  we have
    \[
    \rank_{R} \Jsharp(K_{1}; \Gamma) = 0.
    \]
  It is a finitely-generated torsion module.
\end{proposition}

The web $K_{1}$  does not have an embedded bridge. (Indeed,
its $\SO(3)$ representation variety is non-empty and consists of one fully
irreducible representation whose image in $\SO(3)$ is the octahedral
group.) So Theorem~\ref{thm:non-vanishing-Gamma} tells us that
$\Jsharp(K_{1} ; \Gamma)$ is non-zero. From the corresponding
non-vanishing theorem for $\Jsharp(K_{1})$ proved in \cite{KM-jsharp},
we also learn that
\[
              \dim_{\,\F} \Jsharp(K_{1}) > 0.
\]
The tangled handcuffs are therefore an example where there are
non-trivial differentials in the spectral sequence of
Proposition~\ref{prop:spectral-sequence}. 

\begin{figure}
    \begin{center}
        \includegraphics[scale=.35]{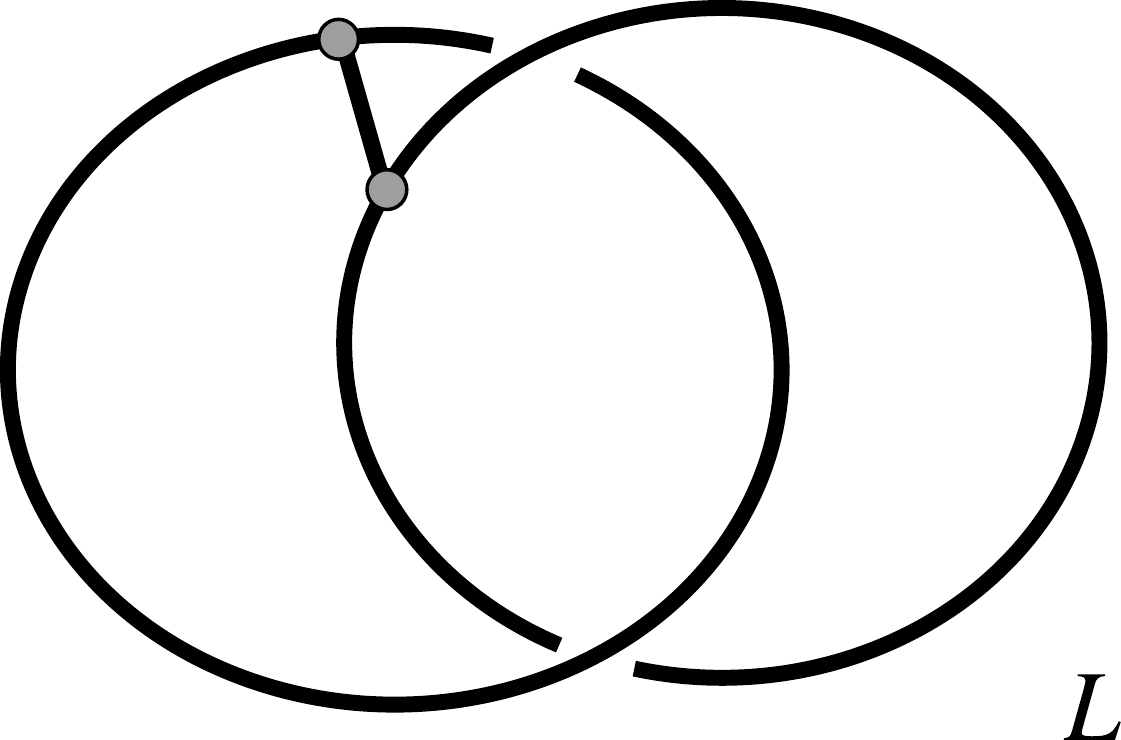}
    \end{center}
    \caption{\label{fig:hopf-cuffs}
   The linked handcuffs, $L$. }
\end{figure}

A related example is the ``linked handcuffs'' $L$ in Figure~\ref{fig:hopf-cuffs}.
There is an exact triangle (Proposition~\ref{prop:exact-triangle}) in
which the role of $L_{2}$ is played by the linked handcuffs and the
role of both $K_{1}$ and $K_{0}$ is played by the unknot $U$. The connecting
homomorphism is provided by a cobordism, $\Sigma$, from one unknot to
the other. This cobordism is obtained by starting with the cylindrical
cobordism $[0,1]\times U$ in $[0,1]\times \R^{3}$ and taking a
connected sum with the pair $(S^{4}, \RP^{2})$, where the $\RP^{2}$ is
standardly embedded in $S^{4}$ with self-intersection $-2$. A gluing
argument shows that the induced map
\[
             \Jsharp(U; \Gamma) \to \Jsharp(U; \Gamma)
\]
is equal to the  map arising from the $2$-dimensional
cohomology class \[\mathsf{z} \in H^{2}(\bonf^{\sharp}(U) ; \F)\] given by
\[
        \mathsf{z} = w_{2}(\mathbb{W}_{s}),
\]
where $\mathbb{W}_{s}$ is the rank-2 bundle \eqref{eq:bpoint-W}
corresponding to the point $s$ on the cylinder where the connected sum
is made. From the Whitney sum formula, we have
\[
                   \mathsf{z} = \mathsf{u}_{s}^{2} +
                   \mathsf{w}_{2,x}(\mathbb{E}_{x}) 
\]
so the corresponding operator on $\Jsharp(U;\Gamma)$ is $u_{e}^{2} + P$.
So there is a short exact sequence
\[
0\to \coker (u_{e}^{2}+P) \to \Jsharp(L ; \Gamma) \to \ker(u_{e}^{2}+P)
\to 0.
\]
The calculation in Proposition~\ref{prop:u-matrix-U} tell us that both
the kernel and cokernel of $u_{e}^{2}+P$ are free of rank $2$. So the
exact sequence of $R$-modules is
\[
0\to  R \oplus R  \to \Jsharp(L ; \Gamma) \to R \oplus R
\to 0,
\]
which necessarily splits. We record this result:

\begin{proposition}\label{prop:hopf-cuffs}
For the linked handcuffs $L$ in Figure~\ref{fig:hopf-cuffs}, we have
\[
        \Jsharp(L ; \Gamma) \cong R^{4}      
\]    
as an $R$-module. \qed
\end{proposition}

This example shows that the rank
of $\Jsharp(L ; \Gamma)$ is not always equal to the number of Tait
colorings for non-planar webs. (The linked handcuffs have no Tait
colorings.) 
With coefficients in the field $\F$, the
original $\Jsharp(L )$ has rank $4$ also, as can be seen using
essentially the same
exact sequence, because the kernel and cokernel of $u_{e}^{2}$ have
dimension $2$ in $\Jsharp(U)$.

\subsection{Vanishing of the first three differentials}

It is not hard to identify the differential on the $E_{1}$
page of the spectral sequence,
Proposition~\ref{prop:spectral-sequence}, arising from the filtration
of the local ring $R_{m}$. Recall that the local system is defined using maps to the
circle,
\[
            h_{i} : \bonf^{\sharp}(\check Y) \to \R/\Z, \qquad i=1,2,3
\]
described at \eqref{eq:h-maps-Y}. If $\mathsf{s}$ is the generator of
$H^{1}(\R/\Z; \F)$, then by pull-back we obtain classes
\[
           \mathsf{t}_{i} \in H^{1}( \bonf^{\sharp}(\check Y) ; \F), 
            \qquad i=1,2,3.
\]
As in section~\ref{sec:op-props}, these classes give rise to operators
\begin{equation}\label{eq:tau-ops}
                 \tau_{i} : \Jsharp(\check Y) \to \Jsharp( \check Y).
\end{equation}
Explicitly, let $h_{i}^{-1}(\delta) \subset \bonf^{\sharp}(\check Y)$ be a generic
level set of $h_{i}$ transverse to all $1$-dimensional moduli spaces
$M_{1}(\alpha,\beta)$. Then the matrix entries of the corresponding chain-map
$\tilde \tau_{i}$ on the
chain level are the mod 2 count of the intersection points
$M_{1}(\alpha,\beta)\cap h_{i}^{-1}(\delta)$. 

\begin{lemma}
    The differential on the $E_{1}$ page the spectral sequence
    of Proposition~\ref{prop:spectral-sequence}
    is the operator
    \[
               d_{1}
                        : \Jsharp(\check Y) \otimes \left(
                            \fm^{p}/\fm^{p+1} \right)
                          \to  \Jsharp(\check Y) \otimes\left( \fm^{p+1}/\fm^{p+2} \right),
    \]
    given by 
    \[
             d_{1} = \sum_{i=1}^{3} \tau_{i} \otimes (1- \bar  T_{i})
   \]
   where $\bar T_{i}$ is the homomorphism $\fm^{p}/\fm^{p+1}
                          \to  \fm^{p+1}/\fm^{p+2}$ given by
                          multiplication by $T_{i}$. In particular,
                          $d_{1}$ is zero if and only if each
                          $\tau_{i}$ is zero.
\end{lemma}

\begin{proof}
    The formula for $d_{1}$ arises by expanding the expression for
    $\partial_{m}$ around $T_{i}=1$. Let us trivialize the local
    system $\Gamma_{m}$ on the complement of $V_{1}\cup V_{2}\cup
    V_{3}$, so that we identify $\Gamma_{m,\beta} = R_{m}$ for all
    critical points $\beta$, and for a path $\zeta$ which is transverse to the
    three $V_{i}$ we have 
    \[
         \Gamma_{m,\zeta}= T_{1}^{n_{1}} T_{2}^{n_{2}} T_{3}^{n_{3}},
     \]
    where $n_{i}$ is the signed intersection number of $\zeta$ with
    $V_{i}$. Modulo $\fm^{2}$, this is equal to
    \[
                 1 + \sum_{1}^{3} \bar{n}_{i} (1-T_{i})
    \]
    where $\bar{n}_{i}$ is the mod 2 residue of $n_{i}$. Thus, at the
    chain level, we have 
    \begin{equation}\label{eq:m-expansion}
              \partial_{m} = \partial + \sum_{1}^{3} \tilde \tau_{i} (1 - T_{i}) + x,
    \end{equation}
    where the matrix entries of $x$ belong to $\fm^{2}$ and $\partial$
    is the ordinary differential on $\Jsharp(\check Y)$, extended to
    the trivial local system with fiber $R_{m}$.
\end{proof}

It turns out that the operators $\tau_{i}$ on $\Jsharp(\check Y)$ are
identically zero. In fact, we have more.

\begin{proposition}\label{prop:d4-ss}
    For any bifold $\check Y$, the differentials $d_{1}$, $d_{2}$,
    $d_{3}$ in the spectral sequence of
    Proposition~\ref{prop:spectral-sequence} are zero.
\end{proposition}

\begin{proof}
Fix an integer $k>0$, and consider the $R$-module \[ \delta_{k} =
m^{k} / m^{k+4}, \] where
$m$ is the maximal ideal at $T_{i}=1$. Let $\Delta_{k}$ be the local
coefficient system obtained from $\Gamma$ by tensoring with this module:
\[
    \Delta_{k} = \Gamma \otimes_{R} \delta_{k}.
\]
For each $k$, we then have instanton homology groups
$\Jsharp (\check Y; \Delta_{k})$. The local system is a system of
$R/ m^{4}$ modules, and this ring has an $m$-adic filtration, of
finite length, as does the module $m^{k} /m^{k+4}$. The latter
filtration has associated graded
\[
      \gr_{p} ( \delta_{k} ) =
      \begin{cases}
          m^{k+p} / m^{k+p+1}, & 0 \le p \le 3 ,\\
           0, & p \ge 4.
      \end{cases}
\]
 For any $\bY$, we have the usual
spectral sequence,
\begin{equation}\label{eq:Delta-ss}
         \Jsharp(\check Y) \otimes \gr (\delta_{k}) \implies \gr
         \Jsharp (\check Y ; \Delta_{k}),
\end{equation} 
and the assertion that the differentials $d_{r}$ are zero for $r\le 3$
in the original spectral sequence of
Proposition~\ref{prop:spectral-sequence} is equivalent to saying that
\emph{all} the differentials in the spectral sequence
\eqref{eq:Delta-ss} are zero. This in turn is equivalent to an
equality of dimensions of finite-dimensional $\F$-vector spaces,
\begin{equation}\label{eq:dimension-Delta-product}
                \dim\Jsharp( \check Y; \Delta_{k} ) = \dim
                \Jsharp( \check Y) \times \dim\delta_{k}.
\end{equation}

To prove the equality \eqref{eq:dimension-Delta-product} and so
complete the proof of Proposition~\ref{prop:d4-ss}, we will draw on the material of
section~\ref{subsec:replace}, so we again consider the orbifold $(S^{3},
H)$ corresponding to the Hopf link, and the marking data $\mu_{H}$
with $w_{2}(E_{\mu})$ non-zero. In order to keep the notation more
compact, we write
\[
      \bH = ((S^{3}, H) ; \mu_{H})
\]
for this marked orbifold, and  we similarly introduce
\[
      \bT = ((S^{3}, \theta) ; \mu_{\theta}).
\]
We write $\bY = (\check Y, \mu_{Y})$ for an arbitrary auxiliary
orbifold, with possibly empty marking. 
 We consider the instanton homology
group
$
        J \bigl( \bH \csum \bT ),
$
and the variant with local coefficients,
$
           J \bigl( \bH \csum \bT ; \Gamma),
$
where the local system as usual is obtained from the marked $(S^{3},
\theta)$ summand by the usual circle-valued functions.

\begin{lemma}
    The Morse complex whose homology is $  J \bigl( \bH \csum \bT ; \Gamma)$ is quasi-isomorphic to a free module of
    rank $4$ with the differential given by
    \[
            \partial_{\,\Gamma} = \begin{pmatrix}   0 & 0 & P & 0 \\
                                                  0 & 0 & 0 & P \\
                                                  0 & 0 & 0 & 0 \\
                                                  0 & 0 & 0 & 0 \end{pmatrix}.
    \]
    In particular, $  J ( \bH \csum \bT ; \Gamma)$ is isomorphic to a sum of two copies of $R/(P)$.
\end{lemma}

\begin{proof}
The proof uses a skein exact sequence, and is very close to the
calculation of $\Jsharp( L ; \Gamma)$ for the ``linked handcuffs'' in
Proposition~\ref{prop:hopf-cuffs}. The skein sequence that holds when
the crossing is entirely contained in the marking region is the usual
skein sequence for links, as developed in \cite{KM-unknot}. We deduce
that the Morse complex which computes $J(\bH \csum \bT ;\Gamma)$ is
quasi-isomorphic to the mapping cylinder of a certain chain map
\begin{equation}\label{eq:U-to-U-Isharp}
               C( \bU \csum \bT ;\Gamma) \to  C( \bU \csum \bT ;\Gamma),
\end{equation}
where $\bU$ is the unknot. The chain map arises from the same
cobordism as in the proof of Proposition~\ref{prop:hopf-cuffs}, namely
the connect sum of the cylindrical cobordism with $(S^{4}, \RP^{2})$,
where the $\RP^{2}$ has self-intersection $-2$. Now, however, the
marking region encompasses the whole of the cobordism, and the marking
data has $w_{2}$ non-zero on the complement of the $\RP^{2}$ in
$S^{4}$. This map \eqref{eq:U-to-U-Isharp} is the operator
corresponding to the class 
\[
        \mathsf{\tilde z} = w_{2}(\tilde{\mathbb{W}}_{s}),
\]
just as in Proposition~\ref{prop:hopf-cuffs}, but now
$\tilde{\mathbb{W}}_{s}$ is the \emph{orientable} rank 2 bundle
corresponding to a base-point on $\bU$. This is the same as the
basepoint operator $w_{2,x}$, which acts by multiplication by $P$.
The instanton homology $ J( \bU \csum \bT ;\Gamma)$ is free of rank
$2$, and arises from a complex with trivial differential, 
because the representation variety is a $2$-sphere. (This is the
calculation of $\Isharp(U)$ from \cite{KM-unknot}.) So the complex
$C( \bH \csum \bT ;\Gamma)$ is quasi-isomorphic to the mapping cone of
multiplication by $P$,
\[
          R \oplus R \stackrel{P}{\to} R \oplus R,
\]
which is what the lemma states.
\end{proof}

\begin{corollary}\label{cor:H-trivial}
   With the local coefficient system $\Delta_{k}$, the instanton
   homology $J \bigl( \bH \csum \bT ; \Delta_{k})$ is the direct sum
   of four copies of the module $\delta_{k}$.
\end{corollary}

\begin{proof}
    The calculation \eqref{eq:PS} shows that the element $P$ belongs
    to the ideal $m^{4} \ideal R$. So the action of $P$ on
    $\delta_{k}$ is zero. The previous lemma therefore tells us that
    the the Morse complex that computes $J \bigl( \bH \csum \bT ;
    \Delta_{k})$ is four copies of the fiber $\delta_{k}$ and the differential is zero.
\end{proof}

To continue the proof of Proposition~\ref{prop:d4-ss},
we now pass from the special case $\bH \csum \bT$
to the general case $\Jsharp(\bY) = J(\bY \csum \bT)$, by excision. As
in section~\ref{subsec:replace}, there are excision cobordisms of
marked bifolds, from (a) to (b) and from (b) to (c), inducing
isomorphisms on $J$ in each case:
\[
\begin{aligned}
   & \text{(a)} & \bY_{a} &= (\bT \csum \bH) \cup (\bT \csum \bY ) \cup (\bH)
   \\ 
   & \text{(b)} & \bY_{b} &=  (\bT ) \cup (\bT  \csum \bH \csum \bY ) \cup (\bH)
   \\ 
    & \text{(c)} & \bY_{c} &=  (\bT ) \cup (\bT  \csum \bH ) \cup (\bH  \csum \bY).
   \\ 
\end{aligned}
\]
In each case, two of the connected components contain a marked theta
graph $\bT$, so for each of (a)--(c) we may define a map to the torus
$T^{3}$ by using the sum of the maps coming from the two copies. In
this way, we have a local system $\Gamma$ over the configuration
spaces in all three cases. The maps obtained from the excision
cobordisms give maps on instanton homology with local coefficients,
and the composite of the two gives an isomorphism
\[
              J (\bY_{a} ; \Delta_{k}) \cong J(\bY_{c}; \Delta_{k})
\]

The contribution of $\bH$ to the calculation of the instanton homology
of $\bY_{a}$ is trivial,
and the Morse complex $C(\bY_{a} ; \Delta_{k})$ that computes  $J
(\bY_{a} ; \Delta_{k})$ can be described as a tensor product,
\[
             C( \bT \csum \bH ; \Delta_{k}) \otimes_{R} C(\bT \csum
             \bH; \Gamma).
\]
Corollary~\ref{cor:H-trivial} therefore gives an isomorphism,
\[
              C (\bY_{a} ; \Delta_{k}) \cong  C(\bT \csum
             \bY;  \Delta_{k})^{\oplus 4}
\]
and hence an isomorphism,
\[
              J (\bY_{a} ; \Delta_{k}) \cong  J(\bT \csum
             \bY;  \Delta_{k})^{\oplus 4}.
\]

On the other hand, the contribution of the first $\bT$ in the
calculation for the instanton homology of $\bY_{c}$ is trivial, and
the local coefficient system comes only from the $(\bT \csum \bH)$
term. So the Morse complex has a description,
\[
                        C (\bY_{c} ; \Delta_{k}) \cong
                        (\delta_{k})^{\oplus 4} \otimes_{\,\F} C(\bH \csum
             \bY).
\]
So we have isomorphisms
\[
\begin{aligned}
    J (\bY_{c} ; \Delta_{k}) &\cong (\delta_{k})^{\oplus 4}
    \otimes_{\,\F} J(\bH \csum \bY) \\
 &\cong (\delta_{k})^{\oplus 4}
    \otimes_{\,\F} J(\bT \csum \bY) .
\end{aligned}
\]
Comparing these expressions for $\bY_{a}$ and $\bY_{c}$, we see that
\[
          \dim  J(\bT \csum
             \bY;  \Delta_{k}) =  \dim J(\bT \csum \bY) \times \dim (\delta_{k}) 
\]
which becomes the desired inequality
\eqref{eq:dimension-Delta-product} when we take $\bY$ to be the
unmarked bifold $\check Y$. This completes the proof of
Proposition~\ref{prop:d4-ss}.
\end{proof}

\bibliographystyle{abbrv}
\bibliography{deformed}

\end{document}